\documentclass[12pt]{amsart}
\usepackage[latin1]{inputenc}
\usepackage{color}
\usepackage{bbm}
\usepackage{amsmath,amssymb}
\usepackage{enumerate}
\usepackage{mathrsfs}
\usepackage{verbatim}
\usepackage[top=2cm,bottom=2cm,left=3cm,right=3cm]{geometry}

\numberwithin{equation}{section}

\newtheorem{theo}{Theorem}[section]
\newtheorem{lem}[theo]{Lemma}
\newtheorem{rem}[theo]{Remark}
\newtheorem{co}[theo]{Corollary}
\newtheorem{prop}[theo]{Proposition}
\newtheorem{fact}[theo]{Fact}

\newtheorem{defn}[theo]{Definition}

\newtheorem{example}[theo]{Example}

\newcommand{\cc}{{\mathscr C}}

\newcommand{\caa}{{\mathscr A}}

\newcommand{\cd}{{\mathscr D}}
\newcommand{\crr}{{\mathscr R}}
\newcommand{\cxx}{{\mathscr X}}

\newcommand{\cpp}{{\mathscr P}}

\newcommand{\cgg}{{\mathscr G}}

\newcommand{\pp}{{\operatorname{p}}} 
\newcommand{\bb}{{\operatorname{b}}} 
\newcommand{\chain}{{\operatorname{ch}}} 

\newcommand{\rr}{\mathbb{R}}

\newcommand{\nn}{\mathbb{N}}

\newcommand{\charfun}{\ensuremath{\mathbbm 1}} 
\newcommand{\sm}[1]{\ensuremath{#1'}}  
\newcommand{\la}[1]{\ensuremath{#1''}} 
\newcommand{\dif}{\,\mathrm{d}}

\DeclareMathOperator{\wI}{w1}
\DeclareMathOperator{\wII}{w2}
\DeclareMathOperator{\wIIs}{w2^*}
\DeclareMathOperator{\Span}{span}
\DeclareMathOperator{\card}{card}

\newcommand{\app}{{\mathcal A}}
\newcommand{\spp}{{\mathcal X}}
\newcommand{\dic}{{\mathcal D}}

\newcommand{\spy}{{\mathcal Y}}

\newcommand{\vc}[1]{{\underline{#1}}}

\title[Properties of local orthonormal systems: 
 Bernstein inequalities]{Properties of local orthonormal systems, \\ Part II: geometric
characterization of Bernstein inequalities}
\author[J. Gulgowski]{Jacek Gulgowski}
\address{Faculty of Mathematics, Physics and Informatics, University of Gda\'nsk,
ul. Wita Stwosza 57, 80-308 Gda\'nsk , Poland}
\email{ jacek.gulgowski@ug.edu.pl}
\author[A.Kamont]{Anna Kamont }
\address{Institute of Mathematics, Polish Academy of Sciences, Branch in Gda\'nsk, ul. Abrahama 18, 81-825 Sopot, Poland}
\email{Anna.Kamont@impan.pl}

\author[M. Passenbrunner]{Markus Passenbrunner}
\address{Institute of Analysis, Johannes Kepler University Linz, Austria, 4040 Linz, Altenberger Strasse 69}
\email{markus.passenbrunner@jku.at, markus.passenbrunner@gmail.com}



\date{\today}

\begin{document}

\begin{abstract}
Let $(\Omega,\mathscr F,\mathbb P) $ be a probability space and let $(\mathscr F_n)_{n=0}^\infty$ be 
a binary  filtration.
i.e. exactly one atom of $\mathscr F_{n-1}$ is divided into 
\emph{two} atoms of $\mathscr F_n$ without any restriction on their respective 
measures. Additionally, denote the collection of atoms corresponding to this 
filtration by $\mathscr A$.
Let 
$S \subset L^\infty(\Omega)$ be a finite-dimensional linear subspace, having an additional stability property on atoms $\caa$.
For these data, we consider two dictionaries:
\begin{itemize}
\item $ \cc = \{ f \cdot \charfun_A: f \in S, A \in \caa\}$, 
\item $\Phi$ -- a local orthonormal system generated by $S$ and the filtration $(\mathscr F_n)_{n=0}^\infty$. 
\end{itemize}
Let $L^p(S) = \overline{{\rm span}} _{L^p(\Omega)}
 \cc =  \overline{{\rm span}}_{L^p(\Omega)} \Phi$,  with $1< p < \infty$.
We are interested in approximation spaces $\app_q^\alpha (L^p(S), \cc)$ and 
$\app_q^\alpha (L^p(S), \Phi)$,  corresponding to the best $n$-term approximation in $L^p(S) $ by elements of $\cc$ and $\Phi$, respectively, where $\alpha >0$ and $0< q \leq \infty$.  
It is known that in the classical Haar case, i.e. when  $S = {\rm span} ( \charfun_{[0,1]})$
 and the binary filtration  
$(\mathscr F_n)_{n=0}^\infty$ is dyadic (that is, an atom $A \in \caa$ is divided into two new atoms of equal measure), we have 
$\app_q^\alpha (L^p(S), \Phi) =
 \app_q^\alpha (L^p(S), \cc)$,
cf. P.~Petrushev \cite{pp.2003.a}.
This motivates us to ask the question whether this equality is true in the general setting described above.
The answer to this question is governed by the validity of a specific 
  Bernstein 
type inequality $ \operatorname{BI}(\caa, S, p, \tau )$,  with parameters $1<p<\infty$, $0<\tau < p$.

The main result of this paper is a geometric characterization of this type of Bernstein inequality $\operatorname{BI}(\caa, S, p, \tau) $,
i.e. a characterization in terms  of the  behaviour of functions from the space $S$ on atoms $\caa$ and rings $\crr = \{ A \setminus B: A, B \in \caa, B \subset A \}\setminus \mathscr A$. 
We specialize this general result to some examples of interest, 
including general Haar systems and spaces $S$ consisting of (multivariate) polynomials.
\end{abstract}

\maketitle

\section{Introduction}
\label{sec:intro}
In the past years, various methods of nonlinear approximation of functions of several variables have attracted much attention. Let us recall some results in this direction

R.A. DeVore, V. Popov \cite{dvp.1987} studied the spaces corresponding to best $n$-term 
approximation of functions from $L^p[0,1]^d$, $0<p<\infty$ by piecewise constant functions, where {\em piecewise constant} means a linear combination of 
characteristic functions of dyadic cubes in $[0,1]^d$.  For suitable range of parameters, they have obtained a description of these spaces:   for some special choice of parameters, best approximation spaces in question  coincide with Besov spaces, while in general, they are identified as real interpolation spaces between $L^p$ and Besov spaces.
 Then, R.A. DeVore, B. Jawerth, V. Popov \cite{dvjp.1992} extended the results of \cite{dvp.1987} to best $n$-term approximation  by wavelet expansions in $L^p(\rr^d)$.  It can be noted that the direct result in \cite{dvjp.1992} is essentially obtained by   greedy approximation, i.e. by   taking the $n$ terms of the  wavelet expansion of the function under consideration with biggest $L^p$ norm. These results can be found also in an early review article R.A. DeVore \cite{dv.1998}.

 A. Cohen, R.A. DeVore, P. Petrushev, H. Xu \cite {cdpx.1999} studied  piecewise constant and Haar thresholding approximation for functions from  $BV(\rr^2)$. An important step in their study was  a general adaptive algorithm  for approximation of subadditive set functions defined on dyadic cubes and rings.
These results were extended to the case of general $d \geq 2$ by
P. Wojtaszczyk \cite{pw.2003}.

Y. Hu, K. Kopotun, X. Yu  \cite{hky.2000} applied the algorithm from \cite{cdpx.1999} to  study spaces of best $n$-term approximation by piecewise polynomials on dyadic cubes in $L^p[0,1]^d$, $0<p<\infty$, and obtained a characterization of these spaces  as real interpolation spaces between $L^p[0,1]^d$ and some variant of functions of bounded variation.
Next, P. Petrushev \cite{pp.2003.a}  studied piecewise polynomial approximation in a more general setting, i.e.  on dyadic rectangular partitions (dyadic in the sense that each rectangle is divided into two rectangles of equal measure); his results contain characterization of the spaces in question  as Besov-type spaces (for a special choice of parameters), or real interpolation spaces between $L^p$ and Besov-type spaces. Moreover, he proved that for $1<p<\infty$, the space corresponding to best $n$-term piecewise constant approximation and best $n$-term Haar approximation are the same.

For more results in this direction, including some results for wavelet expansions of functions from $BV$,  see e.g. 
A. Cohen, Y. Meyer, F. Oru \cite{cmo.1998}, A. Cohen, R.A. DeVore, R. Hochmuth \cite{restr.2000}, A. Cohen, W. Dahmen, I. Daubechies, R. DeVore
\cite{bv.2003}, P. Bechler, R. DeVore, A. Kamont, G. Petrova, P. Wojtaszczyk
\cite{bv.2007}, Yu. Brudnyi \cite{yb.2017}.

Further development  resulted in introducing the concept of a greedy basis in S.V. Konyagin, V.N. Temlyakov \cite{kon.tem}.
More in this direction can be found in the monograph V.N. Temlyakov \cite{greedy}, or in the lecture notes V.N. Temlyakov \cite{ vt.2015}. The recent paper F. Albiac, J.L. Ansorena, P. Bern{\'a}, P. Wojtaszczyk \cite{greedy.2021} extends these notions to the setting of quasi-Banach spaces.

In particular, the results by P.Petrushev \cite{pp.2003.a} show that in case of the Haar system and $1<p<\infty$, the spaces 
corresponding to best $n$-term approximation by characteristic functions of dyadic cubes and by the corresponding Haar functions are the same. 
We ask the question whether the same is true in the setting of  arbitrary partitions and corresponding Haar systems.
The answer turns out to be negative. The aim of this paper is to discuss this problem in an even more general setting, which covers also the case of  
piecewise polynomial approximation on rectangular 
partitions, and generalizes the dyadic setting from \cite{pp.2003.a}.

\subsection{Setting of the problem and formulation of the result.}
Let $(\Omega,\mathscr F,\mathbb P) = (\Omega,\mathscr F,|\cdot|)$ be a probability space and let $(\mathscr F_n)_{n=0}^\infty$ be 
a \emph{binary} filtration, meaning that
\begin{enumerate}
	\item $\mathscr F_0 = \{\emptyset,\Omega\}$,
	\item for each $n\geq 1$, $\mathscr F_n$ is generated by $\mathscr F_{n-1}$ and 
			the subdivision of exactly one atom $A_n$ of $\mathscr F_{n-1}$ into
			two atoms $A_{n}', A_n''$ of $\mathscr F_n$ satisfying $|A_n''|\geq |A_n'|>0$.
\end{enumerate}
For a binary filtration $(\mathscr F_n)$, let $\mathscr A_n$ be the collection of all atoms of
$\mathscr F_n$ and set $\mathscr A = \cup_{n} \mathscr A_n$.

Let 
$S$ be a finite-dimensional linear space of $\mathscr F$-measurable, scalar-valued functions on $\Omega$ 
so that
there exist two constants $c_1,c_2\in (0,1]$ so that for each
atom $A\in\mathscr A$ we have the following stability inequality:
		\begin{equation}\label{eq:L1Linfty}
			|\{ \omega\in A : |f(\omega)| \geq c_1\|f\|_{A}
		\}| \geq c_2|A|,\qquad f\in S,
	\end{equation}
	where by $\|f\|_A$ we denote the $L^\infty$-norm of $f$ on the set $A$.
	Observe that this inequality implies in particular that $S \subset
	L^\infty(\Omega)$.
	For a list of  explicit examples of measure spaces $(\Omega,\mathscr F,\mathbb P)$
and vector spaces $S$ satisfying \eqref{eq:L1Linfty},
 we refer to \cite{part1}. We just remark here that if $\Omega = [0,1]^d$ for some positive integer 
$d$, if $\mathbb P$ is the Lebesgue measure and if the atoms contained in $\mathscr A$ are all rectangles (or---more generally---convex sets),
spaces $S$ consisting of polynomials up to a certain degree in general satisfy 
inequality \eqref{eq:L1Linfty}.

Given any subspace $V\subset L^1$ and any measurable set $A\in\mathscr F$, we let 
\[
	V(A) = \{ f\cdot \charfun_A : f\in V\},
\]
where by $\charfun_A$ we denote the characteristic function of the set $A$ defined by 
$\charfun_A(x) = 1$ if $x\in A$ and $\charfun_A(x)=0$ otherwise.
Moreover, for any $\sigma$-algebra $\mathscr G\subset \mathscr F$, we let 
\[
	V(\mathscr G)= \{ f : \Omega\to\mathbb R\;|\;  \text{for each atom $A$ of $\mathscr G$ there exists 
$g\in V$ so that $f\charfun_A = g\charfun_A$} \}.
\]
We also use the abbreviation $V_n := V(\mathscr F_n)$.
Let $P_n$ be the orthoprojector onto $S_n=S(\mathscr F_n)$ for $n\geq 0$ and set $P_{-1}\equiv 0$. Since $S\subset L^\infty(\Omega)$,
the operator $P_n$ can be extended to $L^1(\Omega)$. 
For each integer $n\geq 0$, we define the projector
	$Q_n = P_{n} - P_{n-1}$
	and we choose an arbitrary orthonormal basis $\Phi_n$
	of the range of $Q_n$ and define 
	the orthonormal system $\Phi = \cup_{n=0}^\infty \Phi_n$.
	The collection $\Phi$ is called a \emph{local orthonormal system}, since 
	it is easy to see that functions in the range of $Q_n$ are supported 
	in the set $A_n$.

We now consider the two dictionaries
$ \cc = \{ f \cdot \charfun_A: f \in S, A \in \caa\}$ and 
$\Phi$.
Let $L^p(S) = \overline{{\rm span}}_{L^p(\Omega)}
 \cc =  \overline{{\rm span}}_{L^p(\Omega)} \Phi$, with $1< p < \infty$.
In this article, we investigate the relation 
between the approximation spaces $\app_q^\alpha (L^p(S), \cc)$ and 
$\app_q^\alpha (L^p(S), \Phi)$,  corresponding to the best $n$-term approximation in $L^p(S) $ by elements of $\cc$ and $\Phi$, respectively, where $\alpha >0$ and $0< q \leq \infty$.  
For the definition and further general properties of those approximation spaces, we 
refer to Section~\ref{sec:appr}.
$\mathscr F_n$ is a binary filtration so each element of $\Phi$ can be represented as a sum of two elements of $\cc$, so consequently we have the continuous embedding $\app_q^\alpha (L^p(S), \Phi)
\hookrightarrow \app_q^\alpha (L^p(S), \cc)$.  
It is known that in the classical Haar case, i.e. when  $S = {\rm span} ( \charfun_{\Omega})$ with $\Omega$ being 
a $d$-dimensional rectangle in $\mathbb R^d$ equipped with $d$-dimensional
Lebesgue measure, and the atoms $\mathscr A$ 
consist of  dyadic rectangles (that is, an atom $A \in \caa$ is divided into two new atoms of equal measure, that are both 
again rectangles), we have 
$\app_q^\alpha (L^p(S), \Phi) =
 \app_q^\alpha (L^p(S), \cc)$, 
cf. P.~Petrushev \cite[Theorems 3.3 and 5.3]{pp.2003.a}.
This motivates us to ask the question whether this equality is true in the general setting
as described above.
 We showed in 
\cite{part1} that (the $L^p$-renormalization of) $\Phi$ is a greedy basis in
$L^p(S)$, which implies that the  continuous embedding $\app_q^\alpha (L^p(S), \cc)
\hookrightarrow \app_q^\alpha (L^p(S), \Phi)$ depends on a specific  Bernstein 
type inequality $ \operatorname{BI}(\caa, S, p, \tau )$,  with parameters $1<p<\infty$, $0<\tau < p$ and $\beta = 1/p - 1/\tau$,  in the following way: 
\begin{itemize}
\item  The Bernstein inequality $ \operatorname{BI}(\caa, S, p, \tau )$ is simultaneously a necessary condition 
for embeddings $\app^\alpha_q (L^p(S), \cc) \hookrightarrow \app^\alpha_q  (L^p(S), \Phi)   $  for all $\alpha  > \beta$, $0 < q \leq \infty$, and  a sufficient condition for embeddings  $\app^\alpha_q(L^p(S), \cc) \hookrightarrow \app^\alpha_q (L^p(S), \Phi) $  for all $0< \alpha  < \beta$, $0 < q \leq \infty$.
\end{itemize}

We give the exact definition of the Bernstein inequality $\operatorname{BI}(\mathscr A,S,p,\tau)$
and its relation to approximation spaces $\app_q^\alpha(L^p(S), \mathscr C),\app_q^\alpha(L^p(S), \Phi)$ in Section~\ref{sec:locorth}. 
The main result of this paper is a geometric characterization of the Bernstein inequality $\operatorname{BI}(\caa, S, p, \tau) $,
i.e. a characterization in terms  of the  behaviour of functions from the space $S$ on atoms $\caa$ and rings $\crr = \{ A \setminus B: A, B \in \caa, B \subset A \}\setminus \mathscr A$. (This definition of a ring is a natural extension of the notion of a dyadic ring, as introduced in \cite{cdpx.1999} and applied  e.g. in \cite{hky.2000, pw.2003}.) 

To this end, we  show the following theorem: 

\begin{theo}\label{thm:main}
For every choice of parameters $(\mathscr A,S,p,\tau)$,  
 the validity of the Bernstein inequality 
$\operatorname{BI}(\mathscr A,S,p,\tau)$
 is equivalent to the following condition: 

There exist a number $\rho\in (0,1)$ and a constant $M$ such that 
for each finite sequence $X_0 \supset X_1 \supset \cdots \supset X_n$ of atoms in $\mathscr A$
with $|X_n|\geq \rho |X_0|$ and $X_j\in \{X_{j-1}',X_{j-1}''\}$ for each $j$, we have the inequality
\begin{equation}\label{eq:introw2s}
\Big( \sum_{i=1}^n \| f \charfun_{X_{i-1}\setminus X_i} ) \|_p^\tau \Big)^{1/\tau} 
\leq M \| f \charfun_{X_0\setminus X_n} \|_p,\qquad f\in S.
\end{equation}
\end{theo}

Section~\ref{sec:bernstein} is devoted to the development of the 
tools needed for this  
result (see p.~\pageref{proof:main} for its eventual proof).
There, the aforementioned condition \eqref{eq:introw2s} will be called $\wIIs(\mathscr A,S,p,\tau)$.
In particular, this condition is satisfied if the filtration $(\mathscr F_n)$ is 
regular in the sense that there exists a constant $c>0$ such that 
for all positive integers $n$ we have the inequality $|A_n'| \geq c|A_n|$.
Then, in Sections~\ref{sec:special} and~\ref{sec:poly}, we apply this characterization to specific examples 
including
\begin{itemize}
\item general Haar systems, i.e. the case of $S = {\rm span} ( \charfun_\Omega)$,
\item $\Omega = [0,1]^d$, and the filtration $(\mathscr F_n)_{n=0}^\infty$ is generated by a sequence of rectangular partitions of $\Omega$, and 
$S$ being the space of polynomials of fixed degree $r$ in $d$ variables.
\end{itemize}

\section{Definitions and Preliminaries} \label{sec:not}

\subsection{Best $n$-term approximation spaces and Bernstein inequalities.} \label{sec:appr}
In this section, we summarize classical facts 
about approximation spaces and Bernstein inequalities, as 
given for instance in  \cite[Chapter 7, Sections 5 and 9]{constr.approx}.
Let $(\spp, \| \cdot \|)$ be a Banach space, and let $\dic \subset \spp$ be a subset that is linearly dense in $\spp$; the set $\dic$ is called a \emph{dictionary}.
We are interested in approximation spaces corresponding to the  best approximation by $n$-term linear  combinations of elements of 
$\dic$. That is, let for $n \in \nn$
$$
\Sigma_n = \Sigma_n^\dic = \Big\{ \sum_{j=1}^n c_j x_j : \text{with }x_j \in \dic \text{ and } \text{scalars } c_j \Big\}.
$$
Then, the best $n$-term approximation by linear combinations of $\dic$ is defined as
$$
 E_n(x) = E_n(x,\dic) =  \inf\{\| x - y\|: y \in \Sigma_n\},\qquad x\in \spp.
$$
We are interested in approximation spaces $\app_q^\alpha(\spp, \dic)$ with $0 < q \leq \infty$, $\alpha >0$, defined by
$$
\app_q^\alpha(\spp, \dic) = \{ x \in \spp: \| x \|_{\app_q^\alpha} = \| x \| + \| \{ 2^{n \alpha} E_{2^n}(x), n \geq 0 \} \|_{\ell^q}  < \infty \}.
$$
In the sequel, we will refer to the following fact, 
which is a direct consequence of 
\cite[Chapter~7, Theorem~9.1]{constr.approx}:
\begin{fact}
\label{intro.f1}
Let $0< q, \rho \leq \infty$ and $0 < \alpha < \beta $. Then $\app_q^\alpha(\spp, \dic) = (\spp, \app_\rho^\beta(\spp, \dic) )_{\alpha/\beta,q}$.
\end{fact}
\begin{proof}
It is enough to see that $\spp$ and $\app_\rho^\beta(\spp, \dic)$ satisfy both Jackson and Bernstein inequalities, as  in \cite[Chapter 7, Equations (5.4), (5.5)]{constr.approx}, with 
exponent $\beta$.
Then, it remains to apply  \cite[Chapter~7, Theorem~9.1]{constr.approx}.
\end{proof}

\begin{fact}
\label{intro.f2} Let $(\spy, \| \cdot \|_\spy) \subset (\spp, \| \cdot\|) $ be a subspace such that $\dic \subset \spy$ (here, $\| \cdot \|_\spy$ can be either a norm or a quasi-norm). Fix $0< q \leq\infty$ and $\gamma> 0$.   
If $\app_q^\gamma (\spp, \dic)  \hookrightarrow 
(\spy, \| \cdot \|_\spy)$, then the following Bernstein inequality is satisfied:
$$
\| y \|_\spy \leq C 2^{n \gamma } \| y \| \quad \hbox{ for } \quad y \in \Sigma_{2^n}. 
$$
Consequently, for all $0< \kappa < \gamma$ and $0 < \rho \leq \infty$, we have $\app_\rho^\kappa (\spp, \dic)  
\hookrightarrow (\spp, \spy)_{\kappa/\gamma, \rho}$.
\end{fact}
\begin{proof} By assumption, $\| y \|_\spy \leq C \| y \|_{\app_q^\gamma}$. 
Therefore, for  $y \in \Sigma_{2^n}$ we have
$$
\| y \|_\spy \leq C \| y \|_{\app_q^\gamma} = C \Big( \| y \|  + \Big(\sum_{j=0}^{n-1} \big( 2^{j \gamma} E_{2^j}(y) \big)^q \Big)^{1/q} \Big) \leq  C 2^{n \gamma} \| y \|.
$$
The last statement is a consequence of the Bernstein inequality, 
cf.   \cite[Chapter 7, Theorem 5.1 (ii)]{constr.approx}, 
or the corresponding argument in the proof of 
\cite[Chapter~7, Theorem~9.1]{constr.approx}.
\end{proof}

A particular case is when the dictionary is a greedy basis in $\spp$ (cf. \cite{kon.tem, greedy}). That is, let $\Psi = \{\psi_n, n \in \nn\}$ be a normalized basis in $\spp$,
with $\Psi^* =\{ \psi^*_n, n \in \nn \}$ being its biorthogonal system. For $x \in \spp$ and $n \in \nn$, let $\Lambda_n(x)$ be a set of indices such that the cardinality $\card \Lambda_n(x)$ of 
the set $\Lambda_n(x)$ equals $n$ and
$\min_{j \in \Lambda_n(x) } |\psi_j^*(x)| \geq \max_{j' \not \in \Lambda_n(x) } |\psi_{j'}^*(x)|$. Then  $G_n(x) = \sum_{j \in \Lambda_n(x) } \psi_j^*(x) \psi_j$ is called the $n$-th greedy approximation of $x$. The basis $\Psi$ is called \emph{greedy} if 
$E_n(x) \simeq \| x - G_n(x)\|$ for all $x \in \spp$ and $ n \in \nn$. It is known that a basis is greedy if and only if it is  unconditional and democratic
cf. \cite[Chapter 1.3]{greedy}.

Let $\Psi = (\psi_n)$ be  a greedy basis that additionally satisfies 
the \emph{$p$-Temlyakov property} (cf. \cite[Equation~(1.130)]{greedy}) which means that
for some constant $C$, the 
inequality 
\begin{equation}\label{eq:temlyakov}
C^{-1}(\card \Lambda)^{1/p}\leq \Big\| \sum_{n\in\Lambda} \psi_n \Big\| \leq C (\card \Lambda)^{1/p}	
\end{equation}
is true for every finite index set $\Lambda$. 

In this case, the space $\app_q^\alpha(\spp, \Psi)$ can be described in terms of 
coefficients $\{\psi_n^*(x), n \in \nn \}$ as in 
\cite[Theorem 1.80]{greedy}. We need this result in the following particular case:
\begin{fact}
\label{intro.f3}
Let $\Psi$ be a greedy basis in $(\spp, \| \cdot \|)$ with $p$-Temlyakov property, $0< \tau < p$ and $\beta = 1/\tau - 1/p$.

Then we have
$$\app_\tau^\beta(\spp, \Psi) = \Big\{ x\in \spp : \sum_{n \in \nn} | \psi^*_n(x)|^\tau < \infty \Big\},$$
with equivalence of (quasi-)norms $\| x \|_{\app_\tau^\beta} \simeq  (\sum_{n \in \nn} | \psi^*_n(x)|^\tau )^{1/\tau}$.
\end{fact}
Now, we would like to specialize Fact \ref{intro.f2} to the case when $\spy$  itself 
is an approximation space corresponding to a greedy basis with the  $p$-Temlyakov property.  

\begin{fact}\label{fact:important}
Let $\dic$ be a dictionary in $(\spp, \| \cdot \|)$, and let $\Psi$ be a greedy basis in $(\spp, \| \cdot \|)$ with the $p$-Temlyakov property.
Let $0< \tau < p$ and $\beta = 1/\tau - 1/p$. Consider the following Bernstein inequality:
\begin{equation}
\label{intro.e1}
\Big(\sum_{j \in \nn} | \psi^*_j(y)|^\tau \Big)^{1/\tau} \leq C 2^{n \beta } \| y \| \quad \hbox{ for } \quad y \in 
\Sigma_{2^n}^\dic. 
\end{equation}

Then:
\begin{itemize}
\item[(i)] If the Bernstein inequality \eqref{intro.e1} is satisfied, then  $\app_q^\alpha (\spp, \dic) \hookrightarrow \app_q^\alpha(\spp, \Psi)$
for all $0 < \alpha <\beta $ and $0<q\leq \infty$.
\item[(ii)] If the Bernstein inequality \eqref{intro.e1} is not satisfied, then $\app_q^\gamma (\spp, \dic) \not\hookrightarrow \app_q^\gamma(\spp, \Psi)$
for any $\gamma > \beta $ and $0<q\leq \infty$.
\end{itemize}
\end{fact}
\begin{proof}
To check (i), denote $\spy = \app_\tau^\beta (\spp, \Psi)$. Combining Facts 
\ref{intro.f3} and \ref{intro.f1} with 
\cite[Chapter~7, Theorem~5.1 (ii)]{constr.approx}
 (or with the corresponding argument in the proof of 
  \cite[Chapter~7, Theorem~9.1]{constr.approx}) we find 
for $0<\alpha < \beta$ and $0<q \leq \infty$
$$
\app_q^\alpha (\spp, \dic) \hookrightarrow (\spp, \spy)_{\alpha/\beta, q} = (\spp, \app_\tau^\beta (\spp, \Psi) )_{\alpha/\beta,q} = \app_q^\alpha (\spp, \Psi).
$$

To see part (ii), suppose on the contrary that $\app_q^\gamma (\spp, \dic) \hookrightarrow \app_q^\gamma(\spp, \Psi)$ for some  $\gamma > \beta $ and $0<q\leq \infty$. We shall see that in such case, the Bernstein inequality \eqref{intro.e1} is satisfied. Let $0< \kappa < \gamma$ and $0< \rho \leq \infty$.
 Observe that  by Fact \ref{intro.f2} (with $\spy =  \app_q^\gamma(\spp, \Psi)$) and Fact \ref{intro.f1}, we have 
$$
 \app_\rho^\kappa (\spp, \dic) \hookrightarrow (\spp, \app_q^\gamma(\spp, \Psi) )_{\kappa/\gamma, \rho} = \app_\rho^\kappa (\spp, \Psi).
$$
Specializing this inclusion to $\kappa = \beta$ and $\rho = \tau$, we get $\app_\tau^\beta (\spp, \dic) \hookrightarrow \app_\tau^\beta (\spp, \Psi)$.
Applying again Fact \ref{intro.f2}, this time with $\spy = \app_\tau^\beta (\spp, \Psi)$, and combining it with Fact \ref{intro.f3} we get the Bernstein
inequality \eqref{intro.e1}.
\end{proof}

\subsection{Bernstein type inequalities for local orthonormal systems} \label{sec:locorth}

	We now continue with our general setting and use the notation 
	introduced in Section~\ref{sec:intro}.
	For any subspace $V\subset L^1$, we denote by
	\begin{equation}\label{eq:def_sigma}
		\Sigma_n(V) =  \Big\{ \sum_{i=1}^n v_i\charfun_{A_i} : v_i\in V, A_i\in\mathscr A \Big\}	
	\end{equation}
	the set of functions
	that can be written as the sum of at most $n$ functions that are contained in $V$ locally 
	on atoms.

If $A\in\mathscr A$ is such that it strictly contains another atom from $\mathscr A$, there is a unique index
$n_0 = n_0(A)\geq 1$ so that $A = A_{n_0}\in\mathscr A_{n_0 -1}$. 
Then we set $\sm{A} := \sm{A_{n_0}}$ and $\la{A} := \la{A_{n_0}}$. 
Moreover, we denote $\bb(\sm{A}) = \la{A}$ and $\bb(\la{A}) = \sm{A}$.
Note that
$\{ \sm{A}, \la{A} \} \subset \mathscr A_{n_0}$.
Moreover, set $\caa^* = \caa \setminus \{\Omega\}$, and define
for $A \in \caa^*$ its predecessor $\pp(A)$ as the smallest 
set (atom) contained in $\mathscr A$ that is still a strict superset of $A$. 
Put 
$\caa' = \{A \in \caa^*: A = \sm{\pp(A)}\}$,
 $\caa'' = \{A \in \caa^*: A = \la{\pp(A)}\}$.
For $\lambda\in (0,1)$, we set $\mathscr A(\lambda) = \{A\in\mathscr A^* : |A|\leq \lambda |\pp(A)|\}$.
For $X \in \caa$, denote by $\chain(X) = \{A \in \caa: A \supseteq X\}$ the finite chain of 
atoms larger than or equal to $X$.
	We now give the definition of the Bernstein inequality $\operatorname{BI}(\mathscr A,S,p,\tau)$:

	\begin{defn}[Bernstein inequality]\label{def:bernstein}
	Fix $1<p<\infty$ and $0<\tau<p$ and let $\beta := 1/\tau - 1/p>0$.
		We say that the \emph{Bernstein inequality} $\operatorname{BI}=\operatorname{BI}(\mathscr A, S,p,\tau)$ 	is satisfied if there 
		exists a constant $C$ such that for all positive integers $n$ and all $g\in \Sigma_n(S)$ we 
		have the inequality
		\begin{equation}\label{eq:BI}
			\Big( \sum_{j\geq 0} \| Q_j g\|_p^{\tau}\Big)^{1/\tau} \leq C n^\beta \|g\|_p.
		\end{equation}
	\end{defn}

	Let us, for $A\in \mathscr A$ and  $n := n_0(A)$, use the notation 
	$Q_A = Q_n$, with the convention that if $n_0(A)$ is not defined (i.e.
	the atom $A$ is never split) then 
	$Q_A \equiv 0$. Using this notation, we can write the left hand side in \eqref{eq:BI} as
	\[
		\Big( \sum_{j\geq 0} \| Q_j g\|_p^{\tau}\Big)^{1/\tau}  = 
			\Big( \|P_0 g\|_p^\tau +  \frac{1}{2}\sum_{A\in \mathscr A^*} \| Q_{\pp(A)} g\|_p^{\tau}\Big)^{1/\tau}.
	\]

We now discuss the relation between this definition of the Bernstein inequality and 
 inequality \eqref{intro.e1}, in order to apply the results from Section~\ref{sec:appr}
 to our setting of local orthonormal systems $\Phi = (\varphi_n)_{n=1}^\infty$, introduced 
 in Section~\ref{sec:intro}. For $1<p<\infty$, we consider the renormalized system $\Psi = (\psi_n)_{n=1}^\infty$,
 given by $\psi_n = \varphi_n / \|\varphi_n\|_p$.
In the following we use the symbol $A(t)\lesssim B(t)$ in order to denote the fact that there
exists a constant $C$ depending only on $c_1,c_2$ from 
\eqref{eq:L1Linfty} and the dimension of $S$  so that for all $t$
we have the inequality $A(t) \leq C B(t)$, where $t$ denotes all implicit or
explicit dependencies that the objects $A$ and $B$ might have. Similarly, we use
the notations $A(t)\gtrsim B(t)$ and $A(t)\simeq B(t)$.

We have shown in 
\cite[Equations (7.2) and (7.3)]{part1} that functions $f$ contained in the range of 
$Q_j$ satisfy 
\begin{equation}\label{eq:norm_equiv}
	\|f\|_p \simeq |T|^{1/p-1/2} \|f\|_2,
\end{equation}
where $T = A'$ with $A\in\mathscr A$ such that $Q_A = Q_j$.
Write $Q_j g = \sum_{\ell = 1}^k \langle g, h_\ell^*\rangle h_\ell$ for the functions 
$(h_\ell)_{\ell=1}^k$ from the collection $\Psi$ that are contained in the range of $Q_j$ 
and its biorthogonal system $(h_\ell^*)_{\ell=1}^k$.
Since the functions $(h_\ell)$ are orthogonal to each other,
each functions $h_\ell^*$ is a constant multiple of  $h_\ell$. 
Thus, we have $\|h_\ell\|_p = 1$ for all $\ell= 1,\ldots, k$.
The equivalence \eqref{eq:norm_equiv} allows us to show that for each integer $j\geq 0$, we have 
\[
\| Q_j g \|_p \simeq |T|^{1/p - 1/2} \| Q_j g\|_2 = |T|^{1/p - 1/2} \Big( \sum_{\ell=1}^k |\langle g,h_\ell^*\rangle|^2 \|h_\ell\|_2^2 \Big)^{1/2}
\simeq \Big( \sum_{\ell=1}^k | \langle g,h_\ell^*\rangle |^2\Big)^{1/2},
\]
and therefore, since $k$ is bounded above by the (finite) dimension of $S$, we obtain
\[
\|Q_j g\|_p^\tau \simeq \sum_{\ell=1}^k |\langle g, h_\ell^*\rangle|^\tau,
\]
showing the equivalence of the Bernstein inequality $\operatorname{BI}(\mathscr A,S,p,\tau)$ 
and inequality \eqref{intro.e1} in the setting of local orthonormal systems for $\spp=L^p(S)$.

We have shown in \cite[Corollary 7.3]{part1} that $\Psi$ is greedy in $L^p$ for $1<p<\infty$
by showing its unconditionality in $L^p$ and that it satisfies the $p$-Temlyakov property \eqref{eq:temlyakov}
in $L^p$.
Thus, we can apply the theory summarized in Section~\ref{sec:appr},
in particular Fact~\ref{fact:important}, to deduce that the equality of 
the approximation spaces $\app_q^\alpha(L^p(S),\mathscr C)$ and  $\app_q^\alpha(L^p(S),\Psi)$
is governed by the validity of the Bernstein inequality $\operatorname{BI}(\mathscr A,S,p,\tau)$, as follows:

\begin{theo}[Bernstein inequality]\label{bernst.equiv}
	Fix $1<p<\infty$ and $0<\tau<p$ and let $\beta := 1/\tau - 1/p>0$. Then:
	\begin{itemize}
\item[(A)]  
If the Bernstein inequality $\operatorname{BI}(\caa, S, p, \tau) $ is not satisfied, then  for all $\alpha  > \beta$, $0 < q \leq \infty$
there is $\app^\alpha_q (L^p(S), \Psi) \hookrightarrow \app^\alpha_q  (L^p(S), \cc)     $, but $\app^\alpha_q (L^p(S), \Psi) \neq \app^\alpha_q  (L^p(S), \cc)   $
\item[(B)] If the Bernstein inequality  $\operatorname{BI}(\caa, S, p, \tau) $ is satisfied, then for all $0< \alpha  < \beta$, $0 < q \leq \infty$
there is $ \app^\alpha_q(L^p(S), \cc) = \app^\alpha_q (L^p(S), \Psi) $. 
\end{itemize}
		\end{theo}

Finally, it seems natural to ask if the space $ \app^\alpha_q(L^p(S), \cc) $ has some external description in case when Bernstein inequality $\operatorname{BI}(\caa, S, p, \tau) $ is not satisfied. For this, recall the paper Y. Hu, K. Kopotun, X. Yu \cite{hky.2000}. This paper treats best $n$-term approximation spaces for $L^p [0,1]^d$, $0<p<\infty$ and $\cc = \{f \cdot   \charfun_A: f \in \cpp_r, A \in \cd\}$, where $\cpp_r$ is the space 
of $d$-variate polynomials of degree $r$, and $\cd$ is the collection of dyadic cubes included in $[0,1]^d$. In particular,  \cite[Corollary 8]{hky.2000} describes 
the spaces $ \app^\alpha_q(L^p[0,1]^d, \cc) $ as interpolation spaces between $L^p[0,1]^d$ and some version of a space of functions of bounded variation
$V^r_{\sigma,p}[0,1]^d$. Let us note that this characterization can be extended to  the general setting of spaces $ \app^\alpha_q(L^p(S), \cc) $. 
We discuss this question in a separate note \cite{part3}.

\section{
Geometric conditions for Bernstein inequalities}
\label{sec:bernstein}
	The aim of this section is to give a condition that relies purely on the
	geometry of the filtration $(\mathscr F_n)$ and on the choice of the space $S$
	that is equivalent to the Bernstein inequality $\operatorname{BI}(\mathscr A,S,p,\tau)$ 
	in Definition~\ref{def:bernstein}

We now outline the plan for doing that. In Section~\ref{sec:pwQ} we collect 
some estimates for the operators $Q_{\pp(A)}$. Section~\ref{sec:w1} gives a 
geometric condition called $\wI$ that is equivalent to the Bernstein inequality 
for $n=1$. Section~\ref{sec:w2} gives more conditions $\wII$ and $\wIIs$ 
(Definitions~\ref{def.2} and \ref{def.3}) and 
investigates the relations among $\wI$ and $\wII$, $\wIIs$.
In Section~\ref{sec:rings} we show that  $\wII$ is 
equivalent to the Bernstein inequality for $n=2$.
In Section~\ref{sec:general}, we show our main result that the Bernstein inequality
is equivalent to the purely geometric condition $\wIIs$.
Finally, Section~\ref{sec:stability} treats the question under which assumptions
the conditions $\wI$ and $\wIIs$ are stable under forming linear spans of 
different choices for the spaces $S$. In particular, we need the results 
of this final subsection of Section~\ref{sec:bernstein} in Section~\ref{sec:poly},
where we treat specific examples of $S$, especially polynomial spaces
on $\mathbb R^d$. 

\subsection{Explicit bounds for the operators $Q_{\pp(A)}$}\label{sec:pwQ}
In this section we give estimates for the projector $Q_{\pp(A)}$ with $A\in\mathscr A^*$.

	\begin{lem}\label{lem:proj_subset}
		Let $1\leq p\leq \infty$, $A\in\mathscr A$ and denote by $P_A$ the orthoprojector onto $S(A)$.
		
		Then, for every $f\in L^p$ and measurable $\Gamma\subset A$,
		\[
			\| P_A (f \charfun_\Gamma)\|_p \lesssim 
				\Big( \frac{|\Gamma|}{|A|}\Big)^{1/p'} \|f\charfun_\Gamma\|_{p}
		\]	
		with $p' = p/(p-1)$.
	\end{lem}
	\begin{proof}
		Since the dimension of $S$ is finite, it suffices to give an estimate for 
		$\| \langle f\charfun_\Gamma, \psi\rangle \psi\|_p$ for every $L^2$-normalized 
		function $\psi\in\operatorname{ran} P_A$.
		We next observe that by inequality \eqref{eq:L1Linfty}, we have $\|\psi\|_p \simeq |A|^{1/p - 1/2}$.
		This and H\"older's inequality imply
		\[
			\| \langle f\charfun_\Gamma, \psi\rangle\psi\|_p \lesssim  
				\|f\charfun_\Gamma\|_p \|\psi\charfun_\Gamma\|_{p'} |A|^{1/p - 1/2} \lesssim 
				\Big(\frac{|\Gamma|}{|A|}\Big)^{1/p'} \|f\charfun_\Gamma\|_p,
		\]
		where in the last inequality, we also used $\|\psi\|_\infty\simeq |A|^{-1/2}$ and 
		$1/p-1 = -1/p'$.
	\end{proof}

Let $A\in \mathscr A^*$ and $\psi\in \operatorname{ran} Q_{\pp(A)}$ with $\|\psi\|_2 = 1$.
In \cite{part1} we showed the following bounds for all such functions $\psi$, 
denoting $L = \la{\pp(A)}$ and $T=\sm{\pp(A)}$:
\begin{equation}\label{eq:pw}
	\| \psi\|_{\pp(A)^c} = 0,\qquad \|\psi\|_{T}\lesssim |T|^{-1/2},\qquad
	\|\psi\|_{L} \lesssim \frac{|T|^{1/2}}{|L|}.
\end{equation}
If we assume that $|L| \geq (1-c_2/2)|\pp(A)|$ we have the improved bound
\begin{equation}\label{eq:improved}
	\|\psi\|_{L} \lesssim \varepsilon_T \frac{|T|^{1/2}}{|L|},
\end{equation}
where for atoms $A\in\mathscr A^*$, we set
\[
\varepsilon_A := \sup_{u\in S} \frac{\| u\|_A}{\|u\|_{\pp(A)}} \leq 1.
\]
Note that \eqref{eq:norm_equiv}, H\"older's inequality and the fact that the dimension of $\operatorname{ran} Q_{\pp(A)}$ 
is uniformly bounded implies for $1\leq p\leq \infty$
\begin{equation}\label{eq:Qbded}
	\| Q_{\pp(A)} : L^p \to L^p\| \lesssim 1.
\end{equation}

	\begin{lem}
		\label{lem.23}
		For each $1\leq p\leq \infty$, the following statements are true.
		\begin{itemize}
		\item[(i)] For each $A \in \caa'$, $\Gamma \subset A$ and $f\in L^p$,
		$$
		\| Q_{\pp(A)} (f \charfun_\Gamma) \|_p \lesssim  \Big( \frac{|\Gamma| }{|A|}\Big)^{1/p'} \| f \charfun_\Gamma \|_p.
		$$
		\item[(ii)] For each $A \in \caa''$, $\Gamma \subset A$ and $f\in L^p$,
		$$
		\| Q_{\pp(A)} (f \charfun_\Gamma) \|_p \lesssim   
		\Big( \frac{|\Gamma| }{|A|}\Big)^{1/p'}
		 \Big( \frac{|\bb(A)|}{|A|}\Big)^{1/p} \| f \charfun_\Gamma \|_p.
		$$
		If $|A|\geq (1-c_2/2)|\pp(A)|$, the latter estimate can be improved by 
		an additional factor of $\varepsilon_{\bb(A)}\leq 1$ on the right hand side.
		\end{itemize}
		\end{lem}
		\begin{proof}
		It suffices to give an estimate for $\| \langle f\charfun_{\Gamma}, \psi\rangle \psi\|_p$
		for every $L^2$-normalized function $\psi\in \operatorname{ran} Q_{\pp(A)}$, since the
		dimension of $\operatorname{ran} Q_{\pp(A)}$ is uniformly bounded. Using estimate \eqref{eq:norm_equiv}
		and H\"older's inequality gives us
		\begin{equation}\label{eq:scalar_first}
			\| \langle f\charfun_{\Gamma}, \psi\rangle \psi\|_p\ \lesssim \ 
			\|f\charfun_\Gamma\|_p \|\psi\charfun_{\Gamma}\|_{p'}    |T|^{1/p-1/2}.
		\end{equation}
		If $\Gamma\subset A=T$, using the pointwise estimate \eqref{eq:pw} on $T$, we have
		\[
			\| \psi \charfun_{\Gamma}\|_{p'} \lesssim |T|^{-1/2} |\Gamma|^{1/p'}.
		\]
		On the other hand, if $\Gamma\subset A=L$ and $T=\bb(A)$, using the same estimate on $L$ yields
		\[
			\| \psi \charfun_\Gamma\|_{p'}	\lesssim \frac{|T|^{1/2}}{|L|} |\Gamma|^{1/p'}.
		\]
		Inserting the last two estimates in \eqref{eq:scalar_first} yields (i) and (ii), respectively. 

		If $|A| \geq (1-c_2/2)|\pp(A)|$, we use estimate \eqref{eq:improved} instead
		of \eqref{eq:pw} to obtain the improvement.
		\end{proof}

Introducing the $p$-renormalization $\varepsilon_{A,p}$ of $\varepsilon_A$ for $A\in\mathscr A^*$ defined by 
\begin{equation}\label{eq:eps_equiv}
	\varepsilon_{A,p} := \sup_{u\in S} \frac{\| u\charfun_A\|_p}{\|u\charfun_{\pp(A)}\|_p}
	\simeq \varepsilon_A \Big(\frac{|A|}{|\pp(A)|}\Big)^{1/p},
\end{equation}
we summarize the results of Lemma~\ref{lem.23} as

\begin{co}\label{cor:upperQ}
Let $\lambda = 1-c_2/2 \geq 1/2$ and $A\in\mathscr A^*$. Put
 \begin{equation}\label{eq:def_u}
	u(A) = 	u(A,\mathscr A,S,p) = 
	\begin{cases}
			1,& \text{if }A\in\mathscr A(\lambda), \\	
			\varepsilon_{\bb(A),p},&\text{if } A\notin\mathscr A(\lambda).
		\end{cases}
 \end{equation}
Then, for each $\Gamma\subset A$, we have
\begin{equation}
	\label{eq:Qbound}
	\|Q_{\pp(A)}(f\charfun_\Gamma)\|_p \lesssim u(A) \Big(\frac{|\Gamma|}{|A|} \Big)^{1/p'} \|f\charfun_\Gamma\|_p.
\end{equation}
\end{co}

\subsection{A first geometric condition  related to the Bernstein inequality}\label{sec:w1}

\begin{defn}
		We say that the \emph{Bernstein inequality $\operatorname{BI}_{\text{atoms}} = 
		\operatorname{BI}_{\text{atoms}}(\mathscr A,S,p,\tau)$ for atoms} is 
		satisfied, if \eqref{eq:BI} holds for $n=1$ and all $g\in \Sigma_1(S)$.
\end{defn}
	\begin{defn}
		Let $\mathscr X=(X_i)_{i=1}^n$ be a decreasing sequence of atoms in $\mathscr A$.
		We say that $\mathscr X$ is a \emph{full chain}, if we have 
		\[
			X_{i+1} \in \{ X_i', X_i''\}\text{ for all $i=1,\ldots,n-1$}.
		\]

		If we additionally have the condition $|X_n| \geq \rho|X_1|$ for 
		some $\rho\in (0,1)$, we say that 
		$\mathscr X$ is a \emph{$\rho$-fat full chain}.
	\end{defn}

Given the numbers $u(A)$ from equation \eqref{eq:def_u} 
we formulate the following condition.

\begin{defn}
\label{def.1}
Let  $1<p<\infty $ and $0 < \tau < p$.
We say that condition $\wI=\wI(\caa,S,p,\tau)$ is satisfied if there are 
numbers $\rho\in(0,1)$ 
and $M>0$ such that for each $\rho$-fat full chain 
$\cxx$ we have the inequality 
\begin{equation}
\label{eq.11}
\sum_{A\in \mathscr X : \pp(A)\in\mathscr X} u(A)^\tau \leq M.
\end{equation}
\end{defn}

We will show in this subsection that condition $\wI$ is equivalent to $\operatorname{BI}_{\mathrm{atoms}}$.
In order to do that we begin with the following lemma about the decomposition
of full chains into $\rho$-fat chains.

	\begin{lem}\label{lem:fat_chains}
		For every $\rho \in (0,1)$, every full chain $\mathscr X$ can be decomposed into 
		the union of finitely many 
		$\rho$-fat full chains $\mathscr X_1, \ldots,\mathscr X_k$ satisfying
		\begin{equation}\label{eq:chains}
			\min_{A\in\mathscr X_{s}}|A| < \rho\min_{A\in\mathscr X_{s+1}} |A|,\qquad
			s = 1,\ldots,k-1.	
		\end{equation}
	\end{lem}
	\begin{proof}
		Let $\mathscr X = (X_i)_{i=1}^n$ be a full chain and let $\rho\in (0,1)$.
		Set $i_0 = n$. If $|X_n| \geq \rho|X_1|$, we stop the induction since $\mathscr X$ already 
		is $\rho$-fat. Otherwise, define
		\[
			i_1 := \max\{ j : |X_n| < \rho|X_j|\}.
		\]	
		Then it is clear that $\mathscr X_1 = (X_i)_{i=i_1+1}^n$ is $\rho$-fat. Inductively, we define
		\[
			i_{s+1} = \max\{ j : |X_{i_s}| < \rho|X_j|\},
		\]
		if it exists and we set $i_{s+1}= 0$ otherwise and stop the induction. 
		Since $\mathscr X$ contains only $n$ sets, 
		this induction must terminate at some point (say $i_{k}=0$) and it is clear that 
		the full chains $\mathscr X_s = (X_i)_{i=i_s+1}^{i_{s-1}}$ for $s=1,\ldots,k$ are
		$\rho$-fat by construction.
	 Moreover, by definition of $i_s$,  we also have 
		\begin{equation*}
			\min_{A\in\mathscr X_{s}}|A| = |X_{i_{s-1}}| < \rho |X_{i_{s}}| = 
			\rho\min_{A\in\mathscr X_{s+1}}|A|, \qquad  \text{$s=1,\ldots,k-1$},
		\end{equation*}
		which finishes the proof of the lemma.
	\end{proof}

\begin{co}\label{co:number_fat}
	Let $\mathscr X$ be a $\rho$-fat full chain and let $\rho_0>\rho$. 
	
	Then, $\mathscr X$ can be decomposed into the union of  $k$ $\rho_0$-fat full chains so
	that $k$ satisfies the upper bound $k \leq 1+\log\rho/\log\rho_0$.
\end{co}
\begin{proof}
	Apply Lemma~\ref{lem:fat_chains} with $\rho_0$ so that \eqref{eq:chains}
	and the fact that $\mathscr X = (X_i)_{i=1}^n$ is a $\rho$-fat full chain
	imply
	\[
		\rho |X_1|\leq \min_{A\in\mathscr X_1} |A| < \rho_0^{k-1} \min_{A\in\mathscr X_k} |A| \leq \rho_0^{k-1} |X_1|.
	\]
	This yields the upper bound for $k$.
\end{proof}
\begin{co}\label{cl.1}
Condition w1 does not depend on $\rho$. That is, if there is some $\rho_0\in(0,1)$ 
and $M_0$ such that condition \eqref{eq.11} 
is satisfied with $\rho = \rho_0$ and $M=M_0$, then for each $0<\rho < 1$ there 
is a finite constant $M(\rho)$ such that condition  \eqref{eq.11} 
is satisfied with $\rho$ and $M= M(\rho)$.
\end{co}
\begin{proof}
If  $\rho \geq \rho_0$, then the implication is clear.
If $\rho< \rho_0$, we use Corollary~\ref{co:number_fat}
to get inequality \eqref{eq.11} for $\rho$-fat chains with $M=(1+\log\rho/\log\rho_0)M_0$.
\end{proof}

\begin{lem}\label{lem.2}
Assume that condition $\wI(\mathscr A,S,p,\tau)$ holds for some parameters $1<p<\infty$
and $0<\tau<p$.

Then, the following assertions are true.
\begin{itemize}
\item[(i)] 
For each $\varepsilon>0$ there exists a constant $C= C(\wI, \varepsilon)$ such that for each $X \in \caa^*$
$$
\sum_{A \in \chain(X) \cap \caa^*} u(A)^\tau |A| ^{-\varepsilon} \leq C |X| ^{-\varepsilon}.
$$
\item[(ii)] There exists a constant $C=C(\wI ,p,\tau)$ such that for each $f \in S$, $X \in \caa^*$ and  $\Gamma \subset X$,  
$$
\sum_{A \in \chain(X) \cap \caa^* } \| Q_{\pp(A)}( f \charfun_\Gamma) \|_p^\tau \leq C \| f \charfun_\Gamma \|_p^\tau.
$$
\end{itemize}
\end{lem}
\begin{proof}
	We begin by proving (i) and first fix some number $\rho\in (0,1)$. According 
	to Lemma~\ref{lem:fat_chains}, we split $\chain(X)\cap\mathscr A^*$ into the union of 
	the $\rho$-fat full chains $\mathscr X_1,\ldots,\mathscr X_k$ satisfying
	\begin{equation}\label{eq:geom_chain}
		\min_{A\in\mathscr X_s} |A| < \rho\min_{A\in\mathscr X_{s+1}}|A|, \qquad s=1,\ldots,k-1,
	\end{equation}
	which implies
	\begin{equation}\label{eq:geom}
		 \min_{A\in\mathscr X_{s+1}} |A| \geq \rho^{-s} |X|,\qquad s=0,\ldots,k-1.
	\end{equation}
	Then, we decompose the sum in (i) into those $\rho$-fat full chains and write 
	\[
			\sum_{A \in \chain(X) \cap \caa^*} u(A)^\tau |A| ^{-\varepsilon}
				=\sum_{s=1}^k \sum_{A\in\mathscr X_s} u(A)^\tau |A|^{-\varepsilon},
	\]
	which we estimate
	\begin{align*}
		\sum_{A \in \chain(X) \cap \caa^*} u(A)^\tau |A| ^{-\varepsilon} 
		&\leq   \sum_{s=1}^k \big(\min_{A\in\mathscr X_s}|A|\big)^{-\varepsilon} \sum_{A\in\mathscr X_s} u(A)^\tau  \\
		&\leq  (M+1) |X|^{-\varepsilon}\sum_{s=0}^{k-1} \rho^{\varepsilon s},
	\end{align*}
	where we used \eqref{eq:geom} and condition $\wI$, respectively.
	Summing the geometric series yields (i).

	For (ii), we just use inequality~\eqref{eq:Qbound} and apply (i) with the parameter $\varepsilon = \tau/p'>0$.
\end{proof}
Now we are in the position to prove that $\wI$ is sufficient for the Bernstein inequality for 
atoms.

\begin{prop}
\label{p.1}
Suppose that condition $\wI(\mathscr A,S,p,\tau)$ is satisfied for some parameters. 

Then, the Bernstein inequality for atoms $\operatorname{BI}_{\mathrm{atoms}}(\mathscr A,S,p,\tau)$ is satisfied.
\end{prop}
\begin{proof}
	Let $f\in S$ and let $Y\in\mathscr A$.
	We have $\| P_0 (f \charfun_Y) \|_p \leq  C \| f \charfun_Y \|_p$
	in particular by Lemma~\ref{lem:proj_subset} with the choices $A=\Omega$, $\Gamma=Y$.
	Note that if $Z\in\mathscr A$ is such that either $Y\supseteq Z$ or 
	$Z\cap Y = \emptyset$, then $Q_{Z}(f\charfun_Y)\equiv 0$. By the nestedness
	of the atoms $\mathscr A$, this implies
	\[
		\{ Z\in \mathscr A : Q_{Z}(f\charfun_Y)\not\equiv 0\} \subseteq 
		\{Z\in\mathscr A : Y\subsetneq Z\}\subseteq \pp ( \chain(Y)\cap \mathscr A^*).
	\]	
Therefore, it remains to apply Lemma \ref{lem.2} (ii).
\end{proof}

Before we show that $\wI$ is also necessary for $\operatorname{BI}_{\mathrm{atoms}}$,
we need a few more lemmata.

	\begin{lem}\label{lem:Q_equiv}
		For all $p>1$, there exists a constant $d_1\in (0,1/2]$ so that for all 
		$A\in\mathscr A(d_1)$ and all $u\in S$ we have
		\[
			\| Q_{\pp(A)}(u\charfun_A)\|_p \simeq \|u\charfun_A\|_p.
		\]
	\end{lem}
	\begin{proof}
		The upper bound is a consequence of \eqref{eq:Qbded}. To prove the lower bound,	
observe that for $u\in S$ we have
\begin{equation}\label{eq:decomp1}
	u\charfun_{A} = P_{\pp(A)}(u\charfun_{A}) + Q_{\pp(A)}(u\charfun_{A}).
\end{equation}
By Lemma \ref{lem:proj_subset}, we have the inequality
\[
	\|P_{\pp(A)}(u\charfun_{A})\|_p \lesssim \Big(\frac{|A|}{|\pp(A)|}\Big)^{1/p'} \| u\charfun_{A}\|_{p}.
\]
Since $p'<\infty$, we obtain from \eqref{eq:decomp1} that
 there exists a constant $d_1>0$ so that for  $|A|/|\pp(A)|\leq d_1$, 
	$\|Q_{\pp(A)}(u\charfun_{A})\|_p \gtrsim \|u\charfun_{A}\|_{p}$.
	\end{proof}

\begin{lem} \label{lem:QA_equiv} Let $p\in(1,\infty)$.

	Then, there exists a constant $d\in (0,1/2]$ so that 
	for all $A\in \mathscr A(d)$  
	and for all $B\subset \bb(A)$ with  $|B|\geq (1-d)\cdot|\bb(A)|$,
there exists a function $f_A\in S$ with $\|f_A\|_{\pp(A)} = 1$ so that 
\[
	\|Q_{\pp(A)}(f_A\charfun_B)\|_p \simeq \varepsilon_{A} |A|^{1/p}.
\]
\end{lem}
\begin{proof}
	By Lemma~\ref{lem:Q_equiv}, there exists $d_1\in(0,1/2]$ so that for all $A\in\mathscr A(d_1)$ 
	and all $u\in S$ 
	we have the equivalence
	\[
			\| Q_{\pp(A)}(u\charfun_A)\|_p \simeq \|u\charfun_A\|_p.
	\]	
	Without restriction, we assume that $d_1 \leq c_2/2\leq 1/2$.
Then, since $0 = Q_{\pp(A)}(u\charfun_{\pp(A)}) = Q_{\pp(A)}(u\charfun_{A} ) + Q_{\pp(A)}(u\charfun_{\bb(A)})$ we also have
\begin{equation}\label{eq:Q_lower}
	\|Q_{\pp(A)}(u\charfun_{\bb(A)})\|_p \gtrsim \|u\charfun_{A}\|_{p}\gtrsim \|u\|_{A} |A|^{1/p}. 
\end{equation}
By definition of $\varepsilon_A$, we can choose $u\in S$ such that \eqref{eq:Q_lower} implies
\begin{equation}\label{eq:QA_upper}
	\| Q_{\pp(A)}(u\charfun_{\bb(A)})\|_p \gtrsim \varepsilon_{A} \|u\|_{\pp(A)} |A|^{1/p}.
\end{equation}
Let $\Gamma\subset \bb(A)$. Then by Corollary \ref{cor:upperQ}, since 
$|A| \leq d_1|\pp(A)|\leq (c_2/2)\cdot|\pp(A)|$,
\begin{equation}\label{eq:QA_upper1}
\begin{aligned}
	\|Q_{\pp(A)}(u\charfun_\Gamma)\|_p &\lesssim  \varepsilon_{A,p} \Big( \frac{|\Gamma|}{|\bb(A)|}\Big)^{1/p'} \|u\charfun_\Gamma\|_p \\
	&\lesssim	\varepsilon_{A} \Big(\frac{|A|}{|\bb(A)|} \Big)^{1/p}
	\Big(\frac{|\Gamma|}{|\bb(A)|} \Big)^{1/p'}
	 \|u\|_{\pp(A)} |\Gamma|^{1/p}  \\
	 &= \varepsilon_{A} \frac{ |\Gamma|}{|\bb(A)|} \|u\|_{\pp(A)} |A|^{1/p}.
\end{aligned}
\end{equation}
Therefore, combining this with \eqref{eq:QA_upper},
there exists a constant $d_2>0$ so that for $|\Gamma|/|\bb(A)| \leq d_2$ we also have with $B = \bb(A)\setminus \Gamma$
\begin{equation}\label{eq:QA_concl}
 \|Q_{\pp(A)}(u\charfun_{B})	\|_p \gtrsim  \varepsilon_{A} \|u\|_{\pp(A)} |A|^{1/p}.
\end{equation}
We now choose $d = \min(d_1,d_2)\leq 1/2$ to obtain the lower bound.

The upper bound directly follows from Corollary~\ref{cor:upperQ} in the same 
manner as the upper bound in \eqref{eq:QA_upper1}.
\end{proof}	

\begin{lem}\label{lem:compact}
	Let $r>0$, $A\in\mathscr A$ and let $U$ be the unit sphere of the space $S(A)$ with
	respect to the norm $\|\cdot \|_A$.

	Then, there exists a constant $N$ that depends only on $r$, the dimension of $S$,
	and the constants $c_1,c_2$ from \eqref{eq:L1Linfty} so that $U$ can be covered by 
	$N$ sets with diameter $\leq r$.
	In particular, the constant $N$ does not depend on the atom $A$.
\end{lem}
\begin{proof}
Let $m=\dim S(A)\leq \dim S$ and let
$(\psi_j)_{j=1}^m$ be an orthonormal basis of $S(A)$ in the inner product space $L^2(A)$. Then
$I\big( (a_j)_{j=1}^m\big)  := \sum_{j=1}^m a_j\psi_j |A|^{1/2}$ is an isomorphism from $\ell_2^m$ to
$S(A)$ with constants $C^{-1}$ and $C$,
where $C$ depends only on the constants $c_1,c_2$ from \eqref{eq:L1Linfty}. Therefore 
$U$ is contained in $I(B)$ where $B$ is the ball in $\ell_2^m$ with center $0$ 
and radius $C$. We cover the compact ball $B$ with balls $(B_n)_{n=1}^N$ each having diameter
$r/C$. Thus, $N$ depends only on $c_1,c_2,m,r$. 
Letting $B_n' = I(B_n)$, we know that the sets $B_n'$ have diameter $\leq r$ and
the sets $(B_n')_{n=1}^N$ cover $U$. 
\end{proof}
	\begin{theo}
		\label{thm:w1_nec}
		Assume that $\operatorname{BI}_{\mathrm{atoms}}(\mathscr A, S,p,\tau)$ 
		is satisfied for some parameters.
			
			Then, condition  $\wI(\mathscr A,S,p,\tau)$ is satisfied.
	\end{theo}
	\begin{proof}
Let $\mathscr X=(X_i)_{i=1}^n$ be a $\rho$-fat full chain. Since condition $\wI$ does
not depend on $\rho<1$ by Corollary~\ref{cl.1}, we assume that $\rho \geq 1-d$ with $d\leq 1/2$ 
as in Lemma \ref{lem:QA_equiv}, and simultaneously assume that $\rho > 1-c_2$ (where $c_2$ is taken from \eqref{eq:L1Linfty}). The latter condition guarantees that by assumption \eqref{eq:L1Linfty}
\[
c_1 \|f\|_{X_1} \leq \|f\|_{X_n} \leq \|f\|_{X_1}
\]
for each $f\in S$.
Let $\mathscr X^*$ be the 
chain $(X_i)_{i=2}^{n}$.
We see that $\operatorname{BI}_{\mathrm{atoms}}$ implies the inequality
\begin{equation}
	\label{eq:bernstein1}
	\Big(\sum_{A\in \mathscr X^*} \| Q_{\pp(A)} (f \charfun_{X_n}) \|_p^\tau\Big)^{1/\tau} \lesssim |X_n|^{1/p}
\end{equation}
for every function $f\in S$ with $\|f\|_{X_n} \leq 1$.
We want to find a function $f_0\in S$ with $\|f_0\|_{X_n} = 1$ so that we have a  
lower bound for the left hand side of \eqref{eq:bernstein1} that implies $\wI$. 
For each $A\in \mathscr X^*$, according to Lemma \ref{lem:QA_equiv},  we first choose $f_A\in S$ with 
$\|f_A\|_{X_n} = 1$ so that 
$C_1 \varepsilon_{\bb(A)} |\bb(A)|^{1/p} \leq \|Q_{\pp(A)}(f_A\charfun_{X_n})\|_p$ 
where  $C_1$ is the implicit constant for the lower bound from Lemma~\ref{lem:QA_equiv}.
Corollary~\ref{cor:upperQ} and \eqref{eq:eps_equiv} imply for some constant $C_2$ and all $g\in L^p$ the upper estimate 
$\|Q_{\pp(A)}(g\charfun_{X_n})\|_p \leq C_2 \varepsilon_{\bb(A)} |\bb(A)|^{1/p} \|g\|_{X_n}$.
Thus, if $f\in S$ is chosen so that $\|f-f_A\|_{X_n} \leq C_1 / (2C_2)$ then, by Corollary~\ref{cor:upperQ}, for $A\in \mathscr X^*$,
\begin{equation}\label{eq:QA_lower}
	\begin{aligned}
	\|Q_{\pp(A)}(f\charfun_{X_n})\|_p &\geq 
	\|Q_{\pp(A)}(f_A\charfun_{X_n})\|_p - \| Q_{\pp(A)}\big( (f-f_A)\charfun_{X_n} \big)\|_p \\
	&\geq \frac{C_1}{2} \varepsilon_{\bb(A)} |\bb(A)|^{1/p}.
	\end{aligned}
\end{equation}
We apply Lemma~\ref{lem:compact} with the parameters $ r= C_1/(2C_2)$, $A = X_n$ to get a 
constant $N$ that does not depend on the choice of $X_n$ and so that $(B_i')_{i=1}^N$ is 
a collection of subsets of the unit sphere $U$ in $S(A)$ with respect to the norm $\|\cdot \|_{X_n}$
that cover $U$ and all sets $B_i'$ have diameter $\leq C_1/(2C_2)$.
 Then, we consider the following subsets of the chain $\mathscr X^*$ according to the sets $B_i'$:
 
\[
	E_i = \{ A\in \mathscr X^* : f_A \in B_i' \},\qquad i=1,\ldots,N.
 \]
 Note that $\cup_i E_i = \mathscr X^*$, but the sets $E_i$ are not necessarily disjoint.
 Then, choose $n_0\in \{1,\ldots,N\}$  so that
 \begin{equation}\label{eq:choice_n_0}
	 \sum_{A\in E_{n_0}} \big( \varepsilon_{\bb(A)} |\bb(A)|^{1/p}\big)^\tau \geq \frac{1}{N} 
	 \sum_{A\in \mathscr X^*} \big( \varepsilon_{\bb(A)} |\bb(A)|^{1/p}\big)^\tau.
 \end{equation}
 Let $f_0$ be an arbitrary function in $ U \cap B_{n_0}'$. Then we estimate
 \begin{align*}
	 \sum_{A\in \mathscr X^*}  \| Q_{\pp(A)}(f_0\charfun_{X_n}) \|_p^\tau 
	 &\geq \sum_{A\in E_{n_0}}  \| Q_{\pp(A)}(f_0\charfun_{X_n}) \|_p^\tau \\
	 &\geq \Big(\frac{C_1}{2}\Big)^\tau \sum_{A\in E_{n_0}} 
	 (\varepsilon_{\bb(A)} |\bb(A)|^{1/p})^\tau
 \end{align*}
 by inequality \eqref{eq:QA_lower}. Inequality \eqref{eq:choice_n_0} now yields 
 \begin{equation}\label{eq:QA_lower_final}
	 \Big( \sum_{A\in \mathscr X^*}  \| Q_{\pp(A)}(f_0\charfun_{X_n}) \|_p^\tau \Big)^{1/\tau} 
	 \gtrsim \Big(\sum_{A\in \mathscr X^*}\big( \varepsilon_{\bb(A)} |\bb(A)|^{1/p} \big)^\tau \Big)^{1/\tau}.
 \end{equation}
 Combining this estimate with \eqref{eq:bernstein1} for $f=f_0$,  
 observation \eqref{eq:eps_equiv} and the definition \eqref{eq:def_u} of $u(A)$
 immediately yields $\wI$. 
	\end{proof}

We summarize the results in this subsection in the following 
\begin{theo}\label{thm:w1equiv}
	For every choice of parameters $(\mathscr A,S,p,\tau)$, condition
	$\wI(\mathscr A,S,p,\tau)$ is equivalent to the Bernstein inequality for
	atoms $\operatorname{BI}_{\mathrm{atoms}}(\mathscr A, S,p,\tau)$.
\end{theo}

\subsection{More geometric conditions related to the Bernstein inequality}\label{sec:w2}

If $\mathscr X = (X_i)_{i=1}^n$ is a $\rho$-fat full chain, we denote by $R(\mathscr X)$
the ring corresponding to $\mathscr X$  defined by $X_1 \setminus X_n$.

Let us formulate the following two  definitions:
\begin{defn}
\label{def.2}
Fix a family of atoms $\caa$, a space $S \subset L^\infty$, $1<p<\infty $
 and $0 < \tau < p$.
We say that condition $\wII=\wII(\caa,S,p,\tau)$ is satisfied if there are 
$\rho\in(0,1)$ and $M>0$ 
such that for each $ f \in S$ and  all $\rho$-fat full chains $\mathscr X$ we have
\begin{equation}
\label{eq.21}
\Big( \sum_{X\in\mathscr X : \pp(X)\in\mathscr X} \| Q_{\pp(X)} (f \charfun_{R(\mathscr X)} )
 \|_p^\tau \Big)^{1/\tau} \leq M \| f \charfun_{R(\mathscr X)} \|_p.
\end{equation}
\end{defn}

\begin{defn}
\label{def.3}
Fix a family of atoms $\caa$, a space $S  \subset L^\infty$, $1<p<\infty $ and $0 < \tau < p$.
We say that condition $\wIIs=\wIIs(\caa,S,p,\tau)$ is satisfied if there are $\rho\in(0,1)$
 and $M>0$ such that for each $f \in S$ and all $\rho$-fat full chains $\mathscr X$ we have
\begin{equation}
\label{eq.23}
\Big( \sum_{X\in\mathscr X : \pp(X)\in \mathscr X} \| f \charfun_{\bb(X)} ) \|_p^\tau \Big)^{1/\tau} 
\leq M \| f \charfun_{R(\mathscr X)} \|_p.
\end{equation}
\end{defn}

Observe that $R(\mathscr X) = \bigcup_{X\in \mathscr X : \pp(X)\in \mathscr X} \bb(X)$, 
and the sets $\bb(X)$ appearing in this union 
are pairwise disjoint, 
so $\| f \charfun_{R(\mathscr X)}\|_p^p =  \sum_{X\in \mathscr X: \pp(X)\in\mathscr X}  \| f \charfun_{\bb(X)}\|_p^p$.
Since $\tau<p$, this implies that in general, $\wIIs$ is a non-empty condition.

\begin{lem}\label{cl.2}
Assume that condition $\wI$ is satisfied. Then
condition $\wII$ does not depend on $\rho$. That is, if there are some $0<\rho_0<1$ and $M_0$ such that condition \eqref{eq.21} of Definition \ref{def.2}
is satisfied with $\rho = \rho_0$ and $M=M_0$, then for each $0<\rho < 1$ there is $M(\rho)$ such that condition  \eqref{eq.21} of Definition \ref{def.2}
is satisfied with $\rho$ and $M= M(\rho)$.
\end{lem}
\begin{proof}
If $\rho \geq \rho_0$, the implication is clear.

Let $\rho < \rho_0$ and let $\mathscr X$ be a $\rho$-fat full chain. 
According to Corollary~\ref{co:number_fat}, we decompose $\mathscr X$ into $k\leq 1
+\log\rho/\log\rho_0$
full chains $\mathscr X_1,\ldots,\mathscr X_k$ that are all $\rho_0$-fat and 
that are increasingly ordered, which means in particular that $\mathscr X_1$ contains the smallest set
in $\mathscr X$ and $\mathscr X_k$ contains the largest set in $\mathscr X$.
With $h_X :=Q_{\pp(X)} (f\charfun_{R(\mathscr X)})$, write
\begin{align*}
	\sum_{X\in\mathscr X : \pp(X) \in\mathscr X} \| h_X \|_p^\tau	
	&= \sum_{i=1}^k \sum_{X\in\mathscr X_i : \pp(X)\in\mathscr X_i} \| h_X \|_p^\tau
	  +  \sum_{i=1}^{k-1} \|
	   h_{L_i} \|_p^\tau,
\end{align*}
denoting by $L_i$ the largest set in the chain $\mathscr X_i$.
Using inequality \eqref{eq:Qbded}, we estimate the latter sum to be at most $C_1 k 
\|f \charfun_{R(\mathscr X)} \|_p^\tau$ for some constant $C_1$. Therefore, we 
proceed by estimating the sum $\sum_{X\in\mathscr X_i : \pp(X)\in\mathscr X_i} \| h_X \|_p^\tau$
for fixed $i=1,\ldots,k$.  We decompose $R(\mathscr X)$ disjointly in the following way:
\[
	R(\mathscr X) = (L_k\setminus L_i) \cup R(\mathscr X_i) \cup \Gamma,
\]
for some
ring $\Gamma$ that satisfies $\Gamma \subset X$ for all $X\in \mathscr X_i$. 
Insert this decomposition into
the definition of $h_X$ for $X\in\mathscr X_i$ with $\pp(X)\in\mathscr X_i$ to get
\begin{equation}\label{eq:decomp}
	h_X = Q_{\pp(X)}(f\charfun_{L_k\setminus L_i}) + 
		Q_{\pp(X)}(f\charfun_{R(\mathscr X_i)}) + 
		Q_{\pp(X)}(f\charfun_{\Gamma}).
\end{equation}
Observe that since $\pp(X) \in \mathscr X_i$ (and thus $\pp(X)\subset L_i$), we obtain $Q_{\pp(X)}(f\charfun_{L_k\setminus L_i}) \equiv 0$.
Then, we use estimate \eqref{eq.21} to deduce
\begin{equation}\label{eq:part1}
	\sum_{X\in \mathscr X_i : \pp(X)\in\mathscr X_i} \| Q_{\pp(X)}(f\charfun_{R(\mathscr X_i)}) \|_p^\tau 
	\leq M_0^\tau \| f\charfun_{R(\mathscr X_i)} \|_p^\tau \leq M_0^\tau \|f\charfun_{R(\mathscr X)} \|_p^\tau. 
\end{equation}
Moreover, we use condition $\wI$ and Lemma~\ref{lem.2}, item (ii), to get
\begin{equation}\label{eq:part2}
\sum_{X\in \mathscr X_i : \pp(X)\in\mathscr X_i} \| Q_{\pp(X)}(f\charfun_{\Gamma}) \|_p^\tau \leq C_2 \|f\charfun_\Gamma\|_p^\tau
\leq  C_2 \|f\charfun_{R(\mathscr X)} \|_p^\tau
\end{equation}
for some constant $C_2$.
Combining decomposition \eqref{eq:decomp} and estimates \eqref{eq:part1} and \eqref{eq:part2},
we obtain, for each $i=1,\ldots,k$,
\[
	\sum_{X\in\mathscr X_i : \pp(X)\in\mathscr X_i} \| h_X \|_p^\tau  \leq 2^\tau(C_2+M_0^\tau) 
	\|f \charfun_{R(\mathscr X)}\|_p^\tau.
\]
Summarizing, we eventually have
\[
	\sum_{X\in\mathscr X : \pp(X) \in\mathscr X} \| h_X \|_p^\tau	\leq k2^\tau(C_1 + C_2 + M_0^\tau) \|f \charfun_{R(\mathscr X)}\|_p^\tau,
\]
concluding the proof with the constant $M=M(\rho)$ given by
$M^\tau= 2^\tau(C_1 + C_2 + M_0^\tau) (1+\log\rho/\log\rho_0)$.
\end{proof}

\begin{lem}\label{cl.3}
Condition $\wIIs$ does not depend on $\rho$. That is, if there are some $\rho_0\in(0,1)$ and $M_0$ such that condition \eqref{eq.23} of Definition \ref{def.3}
is satisfied with $\rho = \rho_0$ and $M=M_0$, then for each $\rho\in(0,1)$ there is $M(\rho)$ such that condition  \eqref{eq.23} of Definition \ref{def.3}
is satisfied with $\rho$ and $M= M(\rho)$.
\end{lem}
\begin{proof}
If $\rho \geq \rho_0$, the implication is clear.

Let $\rho < \rho_0$ and let $\mathscr X$ be a $\rho$-fat full chain. According to 
Lemma~\ref{lem:fat_chains}, we decompose $\mathscr X$ into $k\leq 1+ \log \rho/\log \rho_0$
full chains $\mathscr X_1,\ldots,\mathscr X_k$ that are all $\rho_0$-fat
and that are increasingly ordered, which means in particular that $\mathscr X_1$ contains the smallest set
in $\mathscr X$ and $\mathscr X_k$ contains the largest set in $\mathscr X$.
Then we first write
\begin{align*}
	\sum_{X\in\mathscr X: \pp(X)\in \mathscr X}	 \|f \charfun_{\bb(X)}\|_p^\tau
	  = \sum_{i=1}^k \sum_{X\in\mathscr X_i : \pp(X)\in\mathscr X_i} \|f\charfun_{\bb(X)}\|_p^\tau
	   + \sum_{i=1}^{k-1} \|f\charfun_{\bb(L_i)}\|_p^\tau,
\end{align*}
where we denote by $L_i$ the largest set contained in the chain $\mathscr X_i$.
Using now \eqref{eq.23} for the $\rho_0$-fat chains $\mathscr X_1,\ldots,\mathscr X_k$, we obtain
\begin{align*}
	\sum_{X\in\mathscr X: \pp(X)\in \mathscr X}	 \|f \charfun_{\bb(X)}\|_p^\tau
	  \leq  M_0 \sum_{i=1}^k \|f\charfun_{R(\mathscr X_i)}\|_p^\tau
	   + \sum_{i=1}^{k-1} \|f\charfun_{\bb(L_i)}\|_p^\tau.
\end{align*}
Since $R(\mathscr X_i) \subset R(\mathscr X)$ for all $i=1,\ldots,k$ and 
$\bb(L_i) \subset R(\mathscr X)$ for all $i=1,\ldots,k-1$, we further get
\[
	\sum_{X\in\mathscr X: \pp(X)\in \mathscr X}	 \|f \charfun_{\bb(X)}\|_p^\tau 
	\leq (M_0+1)k \cdot \| f\charfun_{R(\mathscr X)} \|_p^\tau,
\]
which implies \eqref{eq.23} for the $\rho$-fat chain $\mathscr X$
 with the constant $M=M(\rho) = (M_0 + 1)(1+\log\rho/\log\rho_0)$.
\end{proof}

\begin{theo}\label{thm:w2s_implies_w1}
	Let $\wIIs(\mathscr A,S,p,\tau)$ be satisfied for some parameters.

	Then, $\wI(\mathscr A,S,p,\tau)$ is satisfied.
\end{theo}
\begin{proof}
Fix $\rho$ with  $\max(1-d,1-c_2/2) \leq \rho < 1$, where $d$ is the constant from
	Lemma \ref{lem:QA_equiv} and $c_2$ is the constant from \eqref{eq:L1Linfty}.
	and let $\mathscr X = (X_i)_{i=0}^n$ be a $\rho$-fat full chain. Denote 
	$R = R(\mathscr X) = X_0\setminus X_n$ and $\mathscr X^* = (X_i)_{i=1}^n$.

	By condition $\wIIs(\mathscr A,S,p,\tau)$, we have for each $f\in S$ the inequality
	\begin{equation}\label{eq:specialw2}
		\Big(\sum_{A\in\mathscr X^*} \|f\charfun_{\bb(A)}\|_p^\tau\Big)^{1/\tau}
		 \leq M \|f\charfun_R\|_p 
		\leq M \|f\charfun_{X_0}\|_p
	\end{equation}
	for some constant $M$. Note that for $A\in\mathscr X^*$, we have 
	$Q_{\pp(A)}(f\charfun_A) = -Q_{\pp(A)}(f\charfun_{\bb(A)})$ and therefore
	\[
		\| Q_{\pp(A)}(f\charfun_A) \|_p = \| Q_{\pp(A)}(f\charfun_{\bb(A)})	\|_p 
		\lesssim \|f\charfun_{\bb(A)}\|_p,
	\]
	where the latter inequality follows from \eqref{eq:Qbded}.  Thus, we use \eqref{eq:specialw2}
	to deduce, for $f\in S$ with $\|f\|_{X_0} = 1$,
	\begin{equation}\label{eq:w1_Q}
		\Big(\sum_{A\in\mathscr X^*} \| Q_{\pp(A)}(f\charfun_A)	\|_p^\tau \Big)^{1/\tau}
		\lesssim M \| f\charfun_{X_0}\|_p \lesssim |X_0|^{1/p}
	\end{equation}
	We now use the same arguments as in the proof of Theorem~\ref{thm:w1_nec}
	with the starting point \eqref{eq:w1_Q} instead of equation \eqref{eq:bernstein1}
	and with $X_n$ replaced by $X_0$.
	This then implies condition $\wI(\mathscr A,S,p,\tau)$.
\end{proof}

\begin{theo}\label{cl.4}
	Let $(\mathscr A,S,p,\tau)$ be some parameters taken by the conditions $\wI, \wII, \wIIs$.

	Then, the following statements are equivalent:
	\begin{enumerate}
		\item $\wI(\mathscr A,S,p,\tau)$ and $\wII(\mathscr A,S,p,\tau)$. 
		\item $\wIIs(\mathscr A,S,p,\tau)$,
	\end{enumerate}
\end{theo}
\begin{proof}
	Since by Theorem~\ref{thm:w2s_implies_w1}, condition $\wIIs$ implies $\wI$, it suffices 
	to show that under condition $\wI$, the conditions $\wII$ and $\wIIs$ are equivalent. 
Since both conditions $\wII$ and $\wIIs$ do not depend on $\rho$, we choose $\rho = 1-d_1$ 
with $d_1$ from Lemma~\ref{lem:Q_equiv}.
Let $\mathscr X = (X_i)_{i=1}^n$ be a $\rho$-fat full chain and denote
$R = R(\mathscr X)$. Then we write, for $X\in \mathscr X$ with $\pp(X)\in\mathscr X$ and $f\in S$,
\begin{align*}
	Q_{\pp(X)}(f\charfun_R) &= Q_{\pp(X)} \big(f (\charfun_{X_1} - \charfun_{X_n})\big)	
	  = Q_{\pp(X)}\big( f(\charfun_{\pp(X)} - \charfun_{X_n}) \big) \\
	  &= Q_{\pp(X)} \big( f(\charfun_X + \charfun_{\bb(X)} - \charfun_{X_n}) \big)  \\
	 & = Q_{\pp(X)}(f\charfun_{\bb(X)}) + Q_{\pp(X)}(f\charfun_{X\setminus X_n}).
\end{align*}
We estimate the latter projection using inequality~\eqref{eq:Qbound} by
\[
	\| Q_{\pp(X)} (f\charfun_{X\setminus X_n})\|_p \leq C u(X) \Big( \frac{|X\setminus X_n|}{|X|}\Big)^{1/p'} \|f\charfun_{X\setminus X_n}\|_p
	\leq C u(X) \|f\charfun_{R}\|_p.
\]
Therefore, under condition $\wI$ (see Definition~\ref{def.1}), we get after summation
\begin{align*}
	\sum_{X\in\mathscr X : \pp(X)\in\mathscr X}&	 \|  Q_{\pp(X)}(f\charfun_{R}) - 
		Q_{\pp(X)}(f\charfun_{\bb(X)}) \|_p^\tau  \\ 
		&=\sum_{X\in\mathscr X : \pp(X)\in\mathscr X} \| Q_{\pp(X)} (f\charfun_{X\setminus X_n}) \|_p^\tau \\
		&\leq C \sum_{X\in\mathscr X : \pp(X)\in\mathscr X} u(X)^\tau \|f\charfun_R\|_p^\tau \leq CM \|f\charfun_R\|_p^\tau.
\end{align*}
Thus, condition $\wII$ is equivalent to the existence of a constant $C$ such that for all $f\in S$
\[
	\sum_{X\in\mathscr X : \pp(X)\in\mathscr X} \|Q_{\pp(X)}(f\charfun_{\bb(X)})\|_p^\tau \leq C \|f\charfun_R\|_p^\tau.
\]
But by Lemma~\ref{lem:Q_equiv}, $\|Q_{\pp(X)}(f\charfun_{\bb(X)})\|_p \simeq \|f\charfun_{\bb(X)}\|_p$ for
all $X\in\mathscr X$ with $\pp(X)\in\mathscr X$, which yields the equivalence of $\wII$ and $\wIIs$
under the assumption $\wI$.
\end{proof}

\subsection{The Bernstein inequality for rings.}\label{sec:rings}
\begin{defn}
		We say that the  \emph{Bernstein inequality $\operatorname{BI}_{\text{rings}} = 
		\operatorname{BI}_{\text{rings}}(\mathscr A,S,p,\tau)$ for rings} is 
		satisfied, if \eqref{eq:BI} holds for $n=1$ and all functions $g$ of 
		the form $g= f\charfun_R$ for some $f\in S$ and some ring $R\in\mathscr R$.
\end{defn}

\begin{prop}
\label{p.2}
For any choice of parameters $\mathscr A,S,p,\tau$, the following statements are
equivalent:
\begin{enumerate}
	\item $\wI(\mathscr A,S,p,\tau)$ and $\wII(\mathscr A,S,p,\tau)$.
	\item  $\operatorname{BI}_{\mathrm{atoms}}(\mathscr A,S,p,\tau)$  and
			$\operatorname{BI}_{\mathrm{rings}}(\mathscr A,S,p,\tau)$,
\end{enumerate}
\end{prop}
\begin{proof}
First we show that (2) implies (1). Indeed, by Theorem~\ref{thm:w1equiv}, 
$\operatorname{BI}_{\mathrm{atoms}}$ implies  $\wI$. Moreover, by comparing the conditions $\wII$ 
and $\operatorname{BI}_{\mathrm{rings}}$, it is apparent that $\operatorname{BI}_{\mathrm{rings}}$ 
implies $\wII$.

Now we show that (1) implies (2).  By Theorem~\ref{thm:w1equiv} again, 
it suffices to show that $\operatorname{BI}_{\mathrm{atoms}}$ and $\wII$ imply $\operatorname{BI}_{\mathrm{rings}}$.
Let $f\in S$ and $R \in \crr$, $R = A \setminus B$ with $B \subset A$, $\pp(B) \neq A$. 
We have $\| P_0 (f \charfun_R) \|_p \leq C \| f \charfun_R \|_p$ in particular
by Lemma \ref{lem:proj_subset} with the choices $A=\Omega$, $\Gamma = R$.
If $Z\in\mathscr A$  with $Z\cap A=\emptyset$ or $Z\subseteq B$ then $Q_{Z}(f\charfun_R) \equiv 0$.
By the nestedness of the atoms, this implies
\[
	\{ Z\in\mathscr A : Q_{Z}(f\charfun_R)\not\equiv 0\} \subseteq \{ Z\in\mathscr A : Z\supsetneq B\} = \pp \big( \{ X\in\mathscr A^* : X\supseteq B\}\big).
\]
We split the latter set into those atoms between $B$ and $A$ and those larger than $A$:
\begin{align*}
	\{ X\in\mathscr A^* : X\supseteq B\} &= 	\{ X\in\mathscr A^* : A\supsetneq X\supseteq B\}
		\cup \{ X\in\mathscr A^* : X\supseteq A\} \\
		&=: \mathscr X_1^*\cup \mathscr X_2
\end{align*}
with the chains $\mathscr X_1 = \{X\in\mathscr A^* : A\supseteq X\supseteq B\}$ and
$\mathscr X_2 = \chain(A)\cap\mathscr A^*$.
By item (ii) of Lemma~\ref{lem.2}, we obtain
\[
	\sum_{X\in\mathscr X_2} \| Q_{\pp(X)}(f\charfun_R)\|_p^\tau	\leq C\|f\charfun_R\|_p^\tau.
\]
It remains to estimate the part 
\begin{equation}
\label{eq.76}
\sum_{X\in\mathscr X_1^*} \| Q_{\pp(X)} (f\charfun_R) \|_p^\tau. 
\end{equation}
To treat this term, we define $\rho := c_2/2$ with the constant $c_2$
from \eqref{eq:L1Linfty} and distinguish two cases.

\textsc{Case 1: $|B| \geq \rho|A|$:} Here, $\mathscr X_1$ is a $\rho$-fat chain.
Since the condition $\wII$ is independent of the value of $\rho$ 
(by Lemma~\ref{cl.2}), 
we use condition $\wII$ to obtain 
\[
	\sum_{X\in\mathscr X_1^*} \| Q_{\pp(X)} (f\charfun_R) \|_p^\tau \leq  M^\tau \|f\charfun_{R}\|_p^\tau.
\]

\textsc{Case 2: $|B| < \rho |A|$:}
Note that $f \charfun_R = f\charfun_A - f \charfun_B$, and consequently
$$
Q_{\pp(X)} (f\charfun_R) = - Q_{\pp(X)}(f \charfun_B), \qquad X\in\mathscr X_1^*.
$$
It follows by $\operatorname{BI}_{\mathrm{atoms}}$ that
\begin{equation}
\label{eq.56}
\sum_{X\in\mathscr X_1^*} \| Q_{\pp(X)} (f\charfun_R) \|_p^\tau = 
\sum_{X\in\mathscr X_1^*} \| Q_{\pp(X)} (f\charfun_B) \|_p^\tau
\leq C^\tau \| f \charfun_B\|_p^\tau \leq  C^\tau \| f \charfun_A\|_p^\tau.
\end{equation}
Since $|R|\geq (1-\rho)|A| = (1-c_2/2)|A|$ and $f\in S$, we use condition \eqref{eq:L1Linfty}
to get $\|f\charfun_A\|_p \lesssim \|f\charfun_R\|_p$. Therefore, combining this with
\eqref{eq.56}, we also get $\operatorname{BI}_{\mathrm{rings}}$ in this case.
\end{proof}

In particular, considering Theorem~\ref{cl.4}, the above proposition shows that 
$\wIIs$ 
is equivalent to the validity of the conditions $\operatorname{BI}_{\mathrm{atoms}}$ and 
$\operatorname{BI}_{\mathrm{rings}}$.

\subsection{The Bernstein inequality in general} \label{sec:general}

To treat the Bernstein inequality in general setting, apart from  
$\Sigma_n(S)$ we also
need to introduce for all positive integers $n$ the spaces 
\[
\Sigma_{n}^{\mathrm{ring}}(S) = \Big\{\sum_{i=1}^n f_i \charfun_{G_i}: f_i \in S,
 G_i \in \caa \cup \crr, i=1,\ldots, n\Big\}, 
\]
with $\{G_i, i =1, \ldots, n\}$ being a family of pairwise disjoint atoms or rings.
Note that the disjointness is not required in the definition \eqref{eq:def_sigma} of $\Sigma_n(S)$

The following property is essential to our treatment of the Bernstein inequality:
\begin{fact}
\label{f.1}
There is a constant $\nu \in \nn $ such that for each $n \in \nn$ there is $\Sigma_n(S) \subset \Sigma_{\nu \cdot n}^{\rm ring}(S)$.
\end{fact}
\begin{proof}
This is a consequence of the combinatorial result \cite[Lemma~4.1]{dvp.1987}.
This lemma is stated for dyadic cubes in $[0,1]^d$, but its proof  uses only 
combinatoric structure of family of dyadic cubes. 
The required representation is essentially contained in 
 \cite[Formula~(4.4)]{dvp.1987}.
 The constant $\nu$ depends only on combinatoric structure of $\caa$. 
 See also \cite[Lemma~3]{hky.2000}. 
\end{proof}

Now, we are ready to formulate the main result of these notes, which is Theorem~\ref{thm:main}.
Observe that with our terminology of Section~\ref{sec:bernstein}, we can rephrase 
the assertion of Theorem~\ref{thm:main} as follows:

\emph{For every choice of parameters $(\mathscr A,S,p,\tau)$,  
 condition $\wIIs(\mathscr A,S,p,\tau)$ is equivalent to 
 the Bernstein inequality 
$\operatorname{BI}(\mathscr A,S,p,\tau)$. Moreover  the Bernstein inequality 
$\operatorname{BI}(\mathscr A,S,p,\tau)$ holds if and only if it holds for $n=1$ and $n=2$.}

\begin{proof}[Proof of Theorem~\ref{thm:main}]\label{proof:main}
By Proposition~\ref{p.2} and Theorem~\ref{cl.4} we know that $\wIIs$ is
equivalent to $\wI$ and $\wII$, or to $\operatorname{BI}_{\mathrm{atoms}}$ and
$\operatorname{BI}_{\mathrm{rings}}$. In the course of the proof we are using
these equivalent versions.

First let us observe that  $\operatorname{BI}$ implies $\operatorname{BI}_{\mathrm{atoms}}$ and  
$\operatorname{BI}_{\mathrm{rings}}$. Indeed $\operatorname{BI}_{\mathrm{atoms}}$ is $\operatorname{BI}$
for $n=1$, while $\operatorname{BI}_{\mathrm{rings}}$ follows from $\operatorname{BI}$ for $n=2$, 
since for $R=A\setminus B$ and $f\in S$ there is
\[
f\chi_R = f\chi_A - f\chi_B \in \Sigma_2^\cc.
\]
To prove the converse implication,
because of Fact \ref{f.1}, it is enough to prove that for each $g \in \Sigma_{ n}^{\mathrm{ring}}(S)$
\begin{equation}
\label{eq.45a}
\Big( \|P_0 g\|_p^\tau +  \sum_{X\in \mathscr A^*} \| Q_{\pp(X)} g\|_p^{\tau}\Big)^{1/\tau}
\leq C n^\beta \| g \|_p.
\end{equation}
For this, it is enough to prove that if $\cgg \subset \caa \cup \crr$ is a  family of pairwise disjoint atoms or rings
with cardinality $n$, then for each choice 
of $f_G\in S, G \in  \cgg $ and $g = \sum_{G \in \cgg} f_G \charfun_G$, we have 
\begin{equation}
\label{eq.45b}
 \sum_{X \in \caa^*} \| Q_{\pp(X)} g \|_p^\tau   \leq C^\tau \sum_{G \in \cgg}  \|  f_G \charfun_G \|_p^\tau.
\end{equation}
Indeed, we have by H\"older's inequality with exponents $q = p/\tau>1$ and $1/q' = 1 - 1/q = 1- \tau/p = \tau \cdot \beta$
\begin{equation}
\label{eq.45c}
 \sum_{G \in \cgg}  \|  f_G \charfun_G \|_p^\tau 
 \leq n^{1/q'} \Big(   \sum_{G \in \cgg}  \|  f_G \charfun_G \|^{p}_p\Big)^{1/q} =
 n^{\tau \cdot \beta} \| g \|_p^\tau.
\end{equation}
Thus, \eqref{eq.45b} implies \eqref{eq.45a}.
The cases of $0< \tau \leq 1$ and $1< \tau <p$ in \eqref{eq.45b} are treated separately.

\textsc{Case 1: $0<\tau\leq 1$.} Clearly, for each $X \in \caa^*$, we have
$$
\| Q_{\pp(X)} g\|_p \leq  \sum_{G \in \cgg}  \| Q_{\pp(X)}( f_G \charfun_G) \|_p.
$$
Since $0<\tau \leq 1$, we get 
$$
\| Q_{\pp(X)} g\|_p^\tau \leq  \sum_{G \in \cgg}  \| Q_{\pp(X)}( f_G \charfun_G) \|_p^\tau.
$$
As $G$ is either an atom or a ring, we invoke Proposition \ref{p.2}
 to get 
$$
\sum_{X \in \caa^*} \| Q_{\pp(X)} g\|_p^\tau \leq  \sum_{G \in \cgg} \sum_{X \in \caa^*}
 \| Q_{\pp(X)}( f_G \charfun_G) \|_p^\tau
\leq C^\tau  \sum_{G \in \cgg} \|  f_G \charfun_G  \|_p^\tau.
$$
That is, we have \eqref{eq.45b} for $0<\tau\leq 1$.

\textsc{Case 2:  $1<\tau < p$. } 
Given $X\in\mathscr A^*$ and $Z = \pp(X)$, we observe that for $G\in\mathscr G$, we
only have $Q_{Z}(f_G\charfun_G)\not\equiv 0$ if $G$ is contained in one of the 
following three sets:
\begin{align*}
	\Lambda'(Z) = \{ G \in \mathscr G :\; & G\subseteq Z'\}, \qquad
	\Lambda''(Z) = \{ G \in \mathscr G : G\subseteq Z''\}, \\
	\tilde{\Lambda}(Z) &= \{ G = A\setminus B\in\mathscr G: B\subsetneq Z\subseteq A\},
\end{align*}
where $A,B\in\mathscr A$ are some atoms. 
Note also that due to the disjointness of the sets in $\mathscr G$, the cardinality of 
the collection $\tilde{\Lambda}(Z)$ is at most one.
According to those sets, we decompose
\begin{align*}
	Q_{Z} g &= 	 \sum_{G\in\Lambda'(Z)} Q_{Z}(f_G\charfun_G) + 
		\sum_{G\in\Lambda''(Z)} Q_{Z}(f_G\charfun_G) + 
		\sum_{G\in\tilde{\Lambda}(Z)} Q_{Z}(f_G\charfun_G) \\
		&=: A_Z(g) + B_Z(g) + C_Z(g).
\end{align*}
Clearly,
\begin{equation}\label{eq:split_A_B_C}
	\sum_{X\in\mathscr A^*} \|Q_{\pp(X)}(g)\|_p^\tau \lesssim
			\sum_{X\in\mathscr A^*} \|A_{\pp(X)}(g)\|_p^\tau + 
			\sum_{X\in\mathscr A^*} \|B_{\pp(X)}(g)\|_p^\tau + 
			\sum_{X\in\mathscr A^*} \|C_{\pp(X)}(g)\|_p^\tau.
\end{equation}
Let's treat the expression $C_{\pp(X)}(g)$ first. Since the cardinality of $\tilde{\Lambda}(\pp(X))$ 
is at most one for each fixed  $X\in\mathscr A^*$, we obtain
\begin{align*}
	\sum_{X\in\mathscr A^*} \|C_{\pp(X)}(g)\|_p^\tau	 &= \sum_{X\in\mathscr A^*} 
	\sum_{G\in\tilde{\Lambda}(\pp(X))} \| Q_{\pp(X)}(f_G\charfun_G)\|_p^\tau \\
	&=\sum_{G\in\mathscr G} 
	\sum_{X\in\mathscr A^* : G\in\tilde{\Lambda}(\pp(X))} \| Q_{\pp(X)}(f_G\charfun_G)\|_p^\tau.
\end{align*}
We use  the Bernstein inequality for rings $\operatorname{BI}_{\mathrm{rings}}$ to continue this estimate and get 
\begin{equation}\label{eq:CX}
	\sum_{X\in\mathscr A^*} \|C_{\pp(X)}(g)\|_p^\tau\lesssim 
	\sum_{G\in\mathscr G} \|f_G \charfun_G\|_p^\tau.
\end{equation}

Next, treat the expression $A_{\pp(X)}(g)$. By the triangle inequality, for fixed $X\in\mathscr A^*$,
\[
	\|A_{\pp(X)}(g)\|_p^\tau	\leq \Big(\sum_{G\in\Lambda'(\pp(X))} \| Q_{\pp(X)}(f_G\charfun_G) \|_p \Big)^\tau.
\]
Use inequality~\eqref{eq:Qbound} to further obtain, with $Y = Y(X) =\pp(X)'$,
\begin{align*}
	\|A_{\pp(X)}(g)\|_p^\tau	&\lesssim u(Y)^\tau \Big(  \sum_{G\in\Lambda'(\pp(X))}
		 \big(|G|/|Y|\big)^{1/p'}\|f_G\charfun_G\|_p \Big)^\tau \\
		 &=u(Y)^\tau \Big(  \sum_{G\in\Lambda'(\pp(X))} \big(|G|/|Y|\big)^{-1/p}
		 \|f_G\charfun_G\|_p\frac{|G|}{|Y|} \Big)^\tau.
\end{align*}
Since the sets in $\mathscr G$ are disjoint and since $G\subset Y$ for $G\in \Lambda'(\pp(X))$, 
we have the inequality  $\sum_{G\in\Lambda'(\pp(X))} |G| \leq |Y|$. Therefore, we use Jensen's inequality to obtain
\[
	\|A_{\pp(X)}(g)\|_p^\tau \lesssim 
		u(Y)^\tau  \sum_{G\in\Lambda'(\pp(X))} \big(|G|/|Y|\big)^{-\tau/p}
		 \|f_G\charfun_G\|_p^\tau\frac{|G|}{|Y|}.
\]
Since $ \varepsilon:=1-\tau/p >0$, this implies
\begin{align*}
		\sum_{X\in\mathscr A^*} \| A_{\pp(X)}(g)\|_p^\tau	&\lesssim
		\sum_{G\in\mathscr G}    |G|^\varepsilon \|f_G\charfun_G\|_p^\tau 
		\sum_{X\in\mathscr A^* : G\in\Lambda'(\pp(X))} |Y(X)|^{-\varepsilon} u(Y(X))^\tau \\
		&\lesssim \sum_{G\in\mathscr G}  \|f_G \charfun_G\|_p^\tau,
\end{align*}
where in the last inequality we used item (i) of Lemma~\ref{lem.2}.
The same line of argument yields the same estimate if we replace $A_{\pp(X)}(g)$ by $B_{\pp(X)}(g)$ by 
just setting $Y(X) = \pp(X)''$ instead.
Combining those estimates with \eqref{eq:CX} and \eqref{eq:split_A_B_C}, we obtain  \eqref{eq.45b}
in the case $1<\tau < p$, finishing the proof of the theorem.
\end{proof}

\subsection{Stability of the conditions \eqref{eq:L1Linfty} and $\wIIs$}\label{sec:stability}
In this section, we investigate the following question:
Let $U_i$, $i=1,\ldots,m$ be finite-dimensional spaces satisfying \eqref{eq:L1Linfty}
and $\wIIs$. 
Under which additional conditions does $S:=\Span\{ U_i : i = 1,\ldots,m\}$ also
satisfy \eqref{eq:L1Linfty} and $\wIIs$?
The answer to this question will be important for considering explicit examples in 
Section~\ref{sec:poly}.

Introduce the following condition ($*_1$): There exists a constant $C$ so that for 
each atom $A\in\mathscr A$ and each $u\in S$, there exists a decomposition 
$ u = u_1 +\cdots + u_m$ with $u_j\in U_j$ so that
\begin{equation}\label{eq:star1}
	\sum_{j=1}^m \| u_j \charfun_A\|_p \leq C\| u\charfun_A\|_p.
\end{equation}
Then we have the following result:
\begin{prop}
	Suppose that for each $j=1,\ldots,m$, 
	we have \eqref{eq:L1Linfty} for $U_j$ and condition $(*_1)$ is satisfied.

	Then, the space $S=\Span\{ U_j : j=1,\ldots,m\}$ satisfies \eqref{eq:L1Linfty} 
	for some constants $c_1,c_2\in (0,1]$.
\end{prop}
\begin{proof}
Fix $A\in \mathscr A$
and $u\in S$ and a decomposition $u=\sum_{j=1}^m u_j$ that satisfies \eqref{eq:star1}.
Then, by \eqref{eq:L1Linfty} for $U_j$ for each $j=1,\ldots,m$,
\[
	\|u_j\charfun_A\|_p \simeq \|u_j\|_A |A|^{1/p},\qquad j=1,\ldots,m.
\]
Therefore we have 
\begin{equation}\label{eq:stab1}
	\|u\|_A |A|^{1/p} \leq \sum_{j=1}^m \|u_j\|_A |A|^{1/p} \simeq \sum_{j=1}^m \|u_j\charfun_A\|_p
	\lesssim \|u\charfun_A\|_p,
\end{equation}
where the latter inequality follows from \eqref{eq:star1}.
This shows \eqref{eq:L1Linfty} for $S$. 
Indeed, assume that
for some $c_1,c_2$, inequality \eqref{eq:L1Linfty} for $S$ is not true, i.e. we have
\[
	| F | \leq c_2|A|
\]
with $F := \{ \omega\in A : |u(\omega)| \geq c_1 \|u\|_A\}$. Then we deduce
\begin{align*}
	\|u\charfun_A\|_p^p &= \int_F |u|^p\dif\mathbb P + \int_{A\setminus F} |u|^p\dif\mathbb P
	\leq |F| \|u\|_A^p + |A\setminus F| c_1^p \|u\|_A^p  \\
	&\leq (c_2 + c_1^p) |A|\|u\|_A^p.
\end{align*}
But this inequality contradicts \eqref{eq:stab1} provided that $c_1,c_2$ are sufficiently
small. Therefore, we have proved the existence of two positive numbers $c_1,c_2$ satisfying
\eqref{eq:L1Linfty} for $S$.
\end{proof}

Next, consider condition $(*_2)$ given by:
There exists a constant $C$ so that for each ring $R\in \mathscr R$ and each $u\in S$, there
exists a decomposition $u = u_1 + \cdots + u_m$ with $u_j\in U_j$ so that
\[
\sum_{j=1}^m \| u_j \charfun_R\|_p \leq C\| u\charfun_R\|_p.	
\]

\begin{prop}\label{prop:stability_weaker}
	Suppose that $\wIIs(\mathscr A,U_j,p,\tau)$ is satisfied for each $j=1,\ldots,m$ and that
	condition $(*_2)$ is satisfied.

	Then, $\wIIs(\mathscr A,S,p,\tau)$ is true.
\end{prop}
This result is a direct consequence of the following, more general proposition.
Observe that condition (1) in Proposition~\ref{prop:stability} below are 
the conditions $\wIIs$ for the functions $u_j$, but only evaluated locally 
on the chain $\mathscr X$. Additionally, condition (2) in Proposition~\ref{prop:stability}
below is condition $(*_2)$, but with the spaces $U_j$ depending also on the ring $R$.
\begin{prop}\label{prop:stability}
	Let $S$ be a finite-dimensional space of scalar-valued functions on $\Omega$.
	Suppose that there exists $\rho\in(0,1)$ and a constant $M$ so that
	for all $\rho$-fat full chains $\mathscr X$ and for all $u\in S$, there exists a decomposition
	$u = u_1 + \cdots + u_m$ with $u_j\in U_j$
	for some subspaces $U_j = U_j(R)$ of $S$ that may depend on $R = R(\mathscr X)$ such that
	the following conditions are true:
	\begin{enumerate}
		\item For all $j=1,\ldots,m$, we have the inequality
					\[
						\sum_{X\in\mathscr X : \pp(X)\in \mathscr X} \|u_j\charfun_{\bb(X)}\|_p^\tau 
							\leq M^\tau \|u_j\charfun_R\|_p^\tau,
					\] 
		\item $\sum_{j=1}^m \|u_j\charfun_R\|_p \leq M \|u\charfun_R\|_p$.
	\end{enumerate}

Then, condition $\wIIs(\mathscr A,S,p,\tau)$ is satisfied.
\end{prop}
\begin{proof}
	Let $\mathscr X$ be a $\rho$-fat full chain.
	Let $R = R(\mathscr X)$ and let $u\in S$. Then, according to
	condition (2), we choose a decomposition
	$u = u_1 + \cdots + u_m$ with $u_j\in U_j(R)$ and  $\sum_{j=1}^m \| u_j \charfun_R\|_p \leq C\| u\charfun_R\|_p$.
	Next we calculate
	\begin{align*}
		\sum_{X\in \mathscr X : \pp(X)\in\mathscr X} \| u\charfun_{\bb(X)} \|_p^\tau
		 \lesssim \sum_{j=1}^m \sum_{X\in\mathscr X: \pp(X)\in\mathscr X} \|u_j\charfun_{\bb(X)}\|_p^\tau.
	\end{align*}
	Now we use condition (1) for each $j$ to estimate further
	\begin{align*}
		\sum_{X\in \mathscr X : \pp(X)\in\mathscr X} \| u\charfun_{\bb(X)} \|_p^\tau
			\leq M^\tau \sum_{j=1}^m \| u_j \charfun_{R} \|_p^\tau	\lesssim M^\tau \Big(\sum_{j=1}^m \|u_j\charfun_R\|_p\Big)^\tau
	\end{align*}
	and finally, according to (2), 
	\[
		\sum_{X\in \mathscr X : \pp(X)\in\mathscr X} \| u\charfun_{\bb(X)} \|_p^\tau \lesssim M^{2\tau} \| u\charfun_R\|_p^\tau,
	\]
	which implies condition $\wIIs(\mathscr A,S,p,\tau)$ since $\mathscr X$ was an arbitrary 
	$\rho$-fat full chain and $u\in S$ was arbitrary.
\end{proof}

\section{Some special cases}\label{sec:special}
We have shown that condition $\wIIs$ (see Definition~\ref{def.3}) 
is equivalent to the Bernstein inequality in 
Theorem~\ref{thm:main} as our main result.
In this section we consider special cases and specific examples for spaces $S$, measure spaces
$\Omega$ and atoms $\mathscr A$ that allow us to even give simpler and more explicit
conditions than $\wIIs$ that are still equivalent to it.
Moreover, we investigate the relations among the conditions $\wIIs$ 
for different parameter choices.

For convenience we recall that condition $\wIIs$ is true if we have the following inequality
for every $\rho$-fat full chain $\mathscr X$ for some $\rho\in(0,1)$:
\begin{equation}\label{eq:wIIs_wh}
\Big( \sum_{X\in\mathscr X : \pp(X)\in \mathscr X} \| f \charfun_{\bb(X)} ) \|_p^\tau \Big)^{1/\tau} 
\leq M \| f \charfun_{R(\mathscr X)} \|_p,\qquad f\in S,
\end{equation}
with the parameters $\mathscr A,S,p,\tau$ satisfying $1<p<\infty$ and $0<\tau<p$ 
and some uniform constant $M$.

\subsection{Regular partitions}
Consider an abstract probability space $(\Omega,\mathscr F,\mathbb P)$,
a binary filtration $(\mathscr F_n)$ and an arbitrary
finite-dimensional space $S$ of scalar-valued functions on $\Omega$
satisfying inequality \eqref{eq:L1Linfty}. Assume that the splitting of $A_n$ into $A_n'$ and $A_n''$ is done in 
such a way that there exists some constant $c_3<1$
	such that for all $n$ we have $|A_n''|\leq c_3|A_n|$.
	Then, condition $\wIIs$ is satisfied for all possible parameter choices of  $p,\tau$, since 
	each $\rho$-fat chain contains only a uniformly bounded number of atoms.

\subsection{The case $\dim S=1$}

Observe that if we assume $\dim S = 1$, condition $\wIIs$ is equivalent to having for each
$\rho$-fat full chain $\mathscr X=(X_i)_{i=1}^n$ and for $f\in S$ with $\|f\|_{X_1} = 1$, 
the inequality
\begin{equation}\label{eq:simple_w2}
	\sum_{X\in\mathscr X : \pp(X)\in\mathscr X}	 \lambda_X^{\sigma} 
	\leq M \Big(\sum_{X\in \mathscr X : \pp(X)\in \mathscr X}\lambda_X\Big)^{\sigma}
\end{equation}
with the setting $\lambda_X = \|f\charfun_{\bb(X)}\|_p^p\simeq \|f\|_{\bb(X)}^p |\bb(X)|$ and $\sigma=\tau/p\in (0,1)$.

\begin{rem}\label{haar}
A more explicit example of this form is if 
$S = \operatorname{span}{\charfun_\Omega}$ and obviously we then also have 
\eqref{eq:L1Linfty} with the constants $c_1=c_2=1$.
This gives even more simplification of the condition $\wIIs(\mathscr A,S,p,\tau)$ which 
then reads that for every $\rho$-fat full chain $\mathscr X$, we have the inequality
\[
\Big( \sum_{X\in\mathscr X : \pp(X)\in \mathscr X} |\bb(X)|^{\tau/p}  \Big)^{1/\tau} 
\leq M \cdot |R(\mathscr X)|^{1/p}.
\]
The corresponding orthonormal functions $\Phi = (\Phi_n)$
are given by the formula  
\begin{equation}\label{eq:expl_haar}
\Phi_n = \frac{|A_n''|^{1/2}}{ (|A_n'|^2 + |A_n'| |A_n''|)^{1/2}} \charfun_{A_{n}'}
- \frac{|A_n'|}{|A_n''|^{1/2} (|A_n'|^2 + |A_n''||A_n'|)^{1/2}}\charfun_{A_n''},
\end{equation}
whose support equals $A_n$. Those functions are unique up to sign. 
In these formulas, we can see the general pointwise estimates \eqref{eq:pw} 
explicitly.
In the case of $|A_n''| = |A_n'|$ we get the familiar expression 
\[
\Phi_n = \frac{1}{|A_n|^{1/2}}(\charfun_{A_n'} - \charfun_{A_n''}).
\]

Observe that generalized Haar systems on the unit interval are of this category. 
\end{rem}

Returning to the general case of $\dim S=1$, we have the following theorem that shows that the conditions $\wIIs(\mathscr A,S,p,\tau)$ 
for the same space $S$ are not equivalent for different values of $\tau$.
\begin{theo}\label{thm:diff_para}
		Fix $\tau_0 < p$.
		If $\dim S=1$ and $(\Omega,\mathscr F,\mathbb P)$ is non-atomic, there exists 
		a filtration $(\mathscr F_n)$ on $(\Omega,\mathscr F,\mathbb P)$ satisfying inequality \eqref{eq:L1Linfty} 
		for all corresponding
		atoms $A\in\mathscr A$ so that 
		\begin{enumerate}[(i)]
			\item $\wIIs(\mathscr A,S,p,\tau)$ is satisfied for $\tau \in (\tau_0,p)$,
			\item $\wIIs(\mathscr A,S,p,\tau_0)$ is not satisfied.
		\end{enumerate} 
\end{theo}

For the proof of this result we first need a couple of lemmata.

\begin{lem}\label{lem:filtr}
	Let $(\Omega,\mathscr F,\mathbb P)$ be a non-atomic  probability space and let $\dim S=1$
	 such that we have the inequality
	\[
			|\{\omega\in \Omega : |f(\omega)| \geq c_1 \| f\|_\Omega \}| \geq c_2 
	\]
	for all $f\in S$.

	Then, for each sequence $(t_n)\subset (0,1/2]$, there exists a binary filtration $(\mathscr F_n)$ 
	so that for each $n$, the following assertions are true:
	\begin{enumerate}[(i)]
		\item  $|A_n'|=t_n|A_n|$ and $|A_n''| = (1-t_n)|A_n|$,
		\item for $B\in \{A_n', A_n''\}$, we have the inequality
		\[
			|\{ \omega\in B : |f(\omega)|\geq c_1\|f\|_{\Omega}\}|\geq c_2|B|,\qquad f\in S.
		\]
	\end{enumerate}

	In particular, this filtration satisfies inequality \eqref{eq:L1Linfty} for all atoms $A\in\mathscr A$.
\end{lem}
\begin{proof}
	Since the dimension of $S$ equals $1$ and conditions (i) and (ii) are
	invariant under rescaling the function $f$, it is enough to show the assertion
	for one function $f\in S$ that satisfies $\|f\|_\Omega  =1$.
	We use induction and assume the inequality
	\[
		|\{ \omega\in A : |f(\omega)|\geq c_1\}|\geq c_2|A|,\qquad A\in\mathscr A_{n-1}.
	\]
	Choose $A_n \in \mathscr A_{n-1}$ arbitrarily and
	set $L = \{ \omega\in A_n : |f(\omega)|\geq c_1\}$ and $R=A_n\setminus L$. 
	Then, choose sets $L'\subset L$ and $R'\subset R$ contained in $\mathscr F$ with the properties
	\[
		|L'| = t_n |L|, \qquad |R'| = t_n |R|, 	
	\]
	which is possible since $\mathscr F$ is non-atomic.	
	Then we define
	\[
		A_n' := L'\cup R',\qquad A_n'' := (L\setminus L') \cup (R\setminus R').	
	\]
	It is easy to see that this setting satisfies (i) and (ii).
\end{proof}

\begin{lem}\label{lem:lambda}
	Fix $\gamma\in (0,1)$.
	There exists a sequence $(\Lambda_n)$ of sequences
	$\Lambda_n=\{\lambda_{n,j}: j=0,\ldots,2^n\}$ of positive numbers and  length $2^n +1$ that 
	has  the following properties:
	\begin{enumerate}[(i)]
		\item  For each $\sigma\in (\gamma,1)$, there exists $C$ such that for all $n$ 
		and all $0\leq k\leq \ell\leq 2^i$, we have
		\[
			\sum_{j=k}^\ell \lambda_{n,j}^\sigma \leq C \Big(\sum_{j=k}^\ell \lambda_{n,j}\Big)^\sigma,
		\]
		\item  We have $\sup_n \sum_{j=0}^{2^n} \lambda_{n,j}^\gamma = \infty$ 
		while $\sum_{j=0}^{2^n} \lambda_{n,j}$ remains bounded as $n$ tends to $\infty$.
	\end{enumerate}
\end{lem}
\begin{proof}
		Let $\mu_n = 2^{-n/\gamma}$, $n\geq 0$, $M_n = 2^n$ and
		\[
		\lambda_{0,0} = \lambda_{0,1} = \mu_0 = 1.
		\]
		We define $\Lambda_n$ inductively and 
	start with $\Lambda_0 = ( \lambda_{0,0}, \lambda_{0,1})$. Having defined $\Lambda_n$ we set
		\[
		\Lambda_{n+1} = (\lambda_{n+1,0}, ...,\lambda_{n+1,M_{n+1}}) ,
		\]
		where $\lambda_{n+1,2k} = \lambda_{n,k}$ for $k=0,...,M_n$ and  $\lambda_{n+1,2k-1} = \mu_{n+1}$  for $k=1,...,M_n$.
		
		Generally we can think of this sequence as given by $\lambda_{n,k} = \mu_{n-s}\leq 1$ where $k=2^s(2r-1)$. We can see that for $s< n$ we have $\mu_{n-s} < 1$.
		
		Now we are going to show (i). 
		For this, fix   $\sigma\in(\gamma,1)$. Note that it is enough to show the following: for all 
		$n\in\nn$ and any $0\leq k\leq \ell\leq 2^n$
		\begin{equation}\label{eq:equiv_lambda}
		\sum_{j=k}^\ell \lambda_{n,j}^\sigma \simeq \max_{k\leq j\leq \ell} \lambda_{n,j}^\sigma,
		\end{equation}
		with constants depending on $\sigma$, but not on $n,k,\ell$. The inequality
		\[
		\max_{k\leq j\leq \ell} \lambda_{n,j}^\sigma \leq \sum_{j=k}^\ell \lambda_{n,j}^\sigma
		\]
		is obvious. Now we are going to work on the reverse inequality and take $0\leq s\leq 2^n$ such that
		\[
		A = \max_{k\leq j\leq \ell} \lambda_{n,j} = \lambda_{n,s}.
		\]
		
		In the case $s\in\{0,2^n\}$, i.e. $A=1$ we have
		\begin{align*}
		\sum_{j=k}^\ell \lambda_{n,j}^\sigma &\leq \sum_{j=0}^{2^n} \lambda_{n,j}^\sigma 
		= 2 + \sum_{\nu=1}^n \mu_\nu^\sigma 2^{\nu-1}  \\
		&=
		2+\frac{1}{2}\sum_{\nu=1}^n 2^{(1-\sigma/\gamma)\nu} \leq 2+\frac{1}{2}\sum_{\nu=1}^\infty 2^{(1-\sigma/\gamma)\nu} = C,
		\end{align*}	
		for some constant $C$ depending on $\sigma$.
		
		In the case $0<s<2^n$ we have $s = 2^\xi (2\eta-1)$ for $0\leq \xi < n$ and $A=\lambda_{n,s} = \mu_{n-\xi}$. Let us look at the indices $s_1<s$ and $s_2>s$ given by
		\[
		s_2 = 2^\xi\cdot 2\eta = 2^{\xi+1}\eta = 2^{\xi_2}(2\eta_1-1),
		\]
		\[
		s_1 = 2^\xi (2\eta-2) = 2^{\xi+1}(\eta-1) = 2^{\xi_1}(2\eta_1-1)
		\]
		the last equality valid if $s_1\neq 0$. Then we have $\xi_1,\xi_2\geq\xi+1$. 
		No matter if $s_1=0$ or $s_1 = 2^{\xi_1}(2\eta_1-1)$ we can see that 
		$0\leq s_1 < s < s_2\leq2^n$,$\lambda_{n,s_1}>\lambda_{n,s}$ and 
		$\lambda_{n,s_2}>\lambda_{n,s}$, what implies that $s_1<k\leq s\leq \ell<s_2$.
		
		Now we are going to estimate the difference $\ell-k<s_2-s_1$, 
		which  is the estimate on the number of summands in the sum 
		$\sum_{j=k}^\ell \lambda_{n,j}^\sigma$:
		\[
		\ell-k<s_2-s_1 = 2^\xi\cdot 2\eta - 2^\xi(2\eta-2) = 2^{\xi+1}.
		\]
		
		Let us remind that $\lambda_{n,j} = \mu_{n-\nu}$ where $2^\nu$ divides $j$ and $2^{\nu+1}$ does not divide $j$.
		This shows that between $k$ and $\ell$ there is exactly one value $j$ such that $2^\xi$ divides $j$ (namely $j=s$) 
		and none such that $2^{\xi+1}$ divides $j$ 
		(this would contradict that $s$ is the index maximizing $\lambda_{n,j}$).
		
		Let us now fix, for a while, $1\leq \nu<\xi$. Then
		\[
		\card\{ j: k\leq j\leq \ell \text{ and }  2^\nu |j \text{ and } 2^{\nu+1}\not|j \} \leq 
		\card\{ j : s_1 < j < s_2 \text{ and } 2^\nu | j \} \leq 2^{\xi-\nu}.
		\]
		Hence we have
		\begin{align*}
		\sum_{j=k}^\ell\lambda_{n,j}^\sigma \leq 
		\sum_{\nu=1}^\xi 2^{\xi-\nu}\mu_{n-\nu}^\sigma
		 &=  \sum_{\nu=1}^\xi 2^{\xi-\nu}2^{-\sigma(n-\nu)/\gamma} =
		2^{\xi-n\sigma/\gamma} \sum_{\nu=1}^\xi 2^{(\sigma/\gamma-1)\nu} \\
		&\simeq 2^{\xi-n\sigma/\gamma} 2^{\xi(\sigma / \gamma - 1)} 
		= 
		\left(2^{-(n-\xi)/\gamma}\right)^\sigma = 
		\mu_{n-\xi}^\sigma = \lambda_{n,s}^\sigma = A^\sigma.
		\end{align*}
		Therefore, we have proven \eqref{eq:equiv_lambda}. Note that in 
		the course of its proof, we only needed the fact that $\sigma > \gamma$.
		This gives that in particular, \eqref{eq:equiv_lambda} is also true 
		for $\sigma = 1$.
		Therefore, using \eqref{eq:equiv_lambda} 
		with the two parameters $\sigma\in (\gamma,1)$ and $1$, respectively, we obtain
		\[
		\sum_{j=k}^\ell \lambda_{n,j}^\sigma \simeq
		\big( \max_{k\leq j\leq \ell} \lambda_{n,j} \big)^\sigma
		 \simeq \Big( \sum_{j=k}^\ell \lambda_{n,j} \Big)^\sigma.
		\]
		It follows that (i) is satisfied.
		
		Next, we are going to show (ii).  We have
		\begin{align*}
		\sum_{j=0}^{2^n} \lambda_{n,j} 
		&= 2+\sum_{\ell=1}^{n} 2^{\ell-1}\mu_\ell  
		= 2 + \frac{1}{2}\sum_{\ell=1}^{n} 2^\ell 2^{-\ell/\gamma} \\
		 &= 2 + \frac{1}{2}  \sum_{\ell=1}^{n} 2^{\ell (1-1/\gamma)} 
		 \leq 2 + \frac{1}{2}  \sum_{\ell=0}^{\infty} \left( 2^{(1-1/\gamma)}\right)^\ell 
		 = 2 + \frac{1}{2}\frac{1}{1-2^{1-1/\gamma}},
		\end{align*}
		which is just some constant (depending only on  $\gamma$).	
		On the other hand 
		\[
		\sum_{j=0}^{n} \lambda_{n,j}^\gamma 
		= 2 + \sum_{\ell=1}^{n} 2^{\ell-1}\mu_\ell^\gamma
		= 2 + \frac{1}{2} \sum_{\ell=1}^{n} 2^\ell(2^{-\ell/\gamma})^\gamma \simeq n,
		\]
		which concludes the proof of the lemma.
\end{proof}

Now we are ready to prove Theorem~\ref{thm:diff_para}.
\begin{proof}[Proof of Theorem~\ref{thm:diff_para}]
	Invoke Lemma~\ref{lem:lambda} with the parameter $\gamma=\tau_0/p$ to get 
	corresponding sequences $\Lambda_i = \{ \lambda_{i,j} : j = 0,\ldots,2^i\}$ for each positive integer $i$. 
	Choose $\rho$ sufficiently large. 
	We choose the sequence $(t_n)$ such that if we apply Lemma~\ref{lem:filtr}, we 
	get a filtration $(\mathscr F_n)$ that, for each positive integer $i$,
	 contains $\rho$-fat full chains $\mathscr X = (X_j)_{j=0}^{2^i}$ so that
	 $|\bb(X_j)| = z_i \lambda_{i,j}$ for each $j$, where $z_i$ is some appropriate scaling factor.
	 Note also that by Lemma~\ref{lem:filtr}, on each atom $A$ of the filtration $(\mathscr F_n)$ we have $\|f\|_A \geq c_1$
	 for $f\in S$ with $\|f\|_\Omega =1$ and therefore, \eqref{eq:L1Linfty} is satisfied.
	 Comparing \eqref{eq:simple_w2} and the result of Lemma~\ref{lem:lambda} gives 
	 us the desired result.
\end{proof}

Thus, looking at Theorem~ \ref{thm:diff_para}, we have shown that 
condition $\wIIs(\mathscr A,S,p,\tau)$ in fact depends on the parameters $\tau$ and $p$.
We mention here that it also depends on the parameters $\mathscr A$ and $S$.
For  examples of polynomial spaces $S$ on rectangular atoms $\mathscr A$
showing this dependence, we refer to Example~\ref{ex:depS}.

\section{Polynomial spaces on intervals and rectangles}
\label{sec:poly}

Let
$\Omega=[0,1]^d$ be the unit cube  in $\mathbb R^d$
and the atoms $\mathscr A$ consist of subintervals (rectangles) of $[0,1]^d$ and $\mathbb P$ is the
$d$-dimensional  Lebesgue measure. 
In this section, we want to give equivalent conditions for $\wIIs$ for certain
spaces of polynomials $\mathscr P$ in the style of Remark \ref{haar}. More precisely, we want a version of condition \eqref{eq:introw2s} from Theorem \ref{thm:main} which uses only measures of sets $X_{i-1} \setminus X_i$. Our argument  specializes Propositions~\ref{prop:stability_weaker} and ~\ref{prop:stability} to polynomial spaces and rectangular partitions.
  Before we discuss explicit choices of $S$,
we investigate condition (2) of Proposition~\ref{prop:stability} and give
an easier-to-handle sufficient condition, provided that the space $S$ satisfies 
condition \eqref{eq:L1Linfty} not only for each atom in the filtration, but
for each rectangle $A$ in $[0,1]^d$.
In this case, a ring $R$, the difference of two rectangles,
is then the union of $2d$ intervals (rectangles) $(I_i)_{i=1}^{2d}$.
If $R = I \setminus J$ for two intervals $I = \prod_{j=1}^d [a_j,b_j]$, 
$J= \prod_{j=1}^d [c_j, d_j]$, the intervals $I_i$ are given by the
following formula for $i = 1,\ldots,d$:
\begin{equation}
	\label{eq:Is}
	\begin{aligned}
	I_{i} &= [a_1,b_1] \times \cdots \times [a_{i-1},b_{i-1}] \times [a_i,c_i]
				\times [a_{i+1},b_{i+1}]	\times \cdots \times [a_d, b_d], \\
	I_{d+i} &=		[a_1,b_1] \times \cdots \times [a_{i-1},b_{i-1}] \times [d_i,b_i]
				\times [a_{i+1},b_{i+1}]	\times \cdots \times [a_d, b_d].
	\end{aligned}
\end{equation}
For $d>1$, the intervals $(I_i)_{i=1}^{2d}$ are not
pairwise disjoint.

\begin{prop}\label{prop:almost_orthogonality}
	Let the atoms $\mathscr A$ be rectangles on the unit cube $\Omega=[0,1]^d$
	and $S$ be a finite-dimensional space of scalar-valued functions on $\Omega$ 
	satisfying \eqref{eq:L1Linfty} for each rectangle 
	such that for each ring $R$ with the above decomposition $R=\cup_{i=1}^{2d} I_i$,	
	there exists a positive integer $r$ and functions $u_1,\ldots,u_r\in S$ and 
	a constant $C$ such that for all $K\in \{I_1,\ldots,I_{2d}\}$
\begin{equation}\label{eq:non_parallel}
	\sum_{j=1}^r \|u_j\charfun_K\|_2^2 \leq C \Big\| \sum_{j=1}^r u_j\charfun_K\Big\|_2^2. 
\end{equation}

Then, this decomposition satisfies (2) of Proposition~\ref{prop:stability} for 
this ring $R$, i.e., for all $p$ there exists a constant $C'$ depending only on $C,d$ 
and $c_1,c_2$ from \eqref{eq:L1Linfty} such that
\begin{equation}\label{eq:prop2}
	\sum_{j=1}^r \|u_j\charfun_R\|_p \leq C' \big\|\sum_{j=1}^r u_j\charfun_R\big\|_p.
\end{equation}
\end{prop}
\begin{proof}
Note that for every $v\in S$ and by \eqref{eq:L1Linfty}, we have
\begin{equation}\label{eq:equiv}
	\| v\charfun_R\|_p^p \simeq 
	\sum_{i=1}^{2d} \| v\charfun_{I_i}\|_p^p
	\simeq \Big(\sum_{i=1}^{2d} \|v\charfun_{I_i}\|_2^2\cdot |I_i|^{2/p-1}  \Big)^{p/2},
\end{equation}
where the implicit constants depend also on $d$.

Now we observe that the assumption \eqref{eq:non_parallel} and \eqref{eq:equiv} imply
\eqref{eq:prop2}. Indeed,
\begin{align*}
	\sum_{j=1}^r \|u_j\charfun_R\|_p &\simeq \Big(\sum_{j=1}^r \sum_{i=1}^{2d} \|u_j\charfun_{I_i}\|_2^2 
	\cdot |I_i|^{2/p-1}   \Big)^{1/2} \\
	&\leq C^{1/2} \Big( \sum_{i=1}^{2d}\|u\charfun_{I_i}\|_2^2\cdot |I_i|^{2/p-1}  \Big)^{1/2} 
	\lesssim C^{1/2} \|u\charfun_R\|_p
\end{align*}
with $u = u_1 + \cdots + u_m$.
\end{proof}

\subsection{Polynomial spaces of fixed coordinate degree}

We now consider examples of some polynomial spaces for the space $S$,
in particular, if $\vc{r} = (r_1,\ldots,r_d)$ is a $d$-tuple of non-negative integers, 
we consider the spaces  
$\mathscr P_{\vc{r}}$ of polynomials of coordinate degree at most 
$r_i$ in direction $i$ for each $i=1,\ldots,d$ on the unit cube $[0,1]^d$.
In particular, if $d=1$ and $r$ is a non-negative integer, 
$\mathscr P_r$ denotes the space of polynomials of degree at most $r$ on
the unit interval $[0,1]$.

For a ring $R = I\setminus J$, we use the decomposition $R = \cup_{i=1}^{2d} I_i$
 described in the 
beginning of Section~\ref{sec:poly}.
We now identify explicit functions $(u_i)$ satisfying  \eqref{eq:non_parallel}
that are tensor products of Bernstein polynomials rescaled to the rectangle $I$.

\begin{prop}\label{prop:uniform_parallelity}
	Let $\vc{r} = (r_1,\ldots, r_d)$ be a tuple of non-negative integers.
	Let $R = I\setminus J$ be some arbitrary ring for some rectangles 
	$I = I^1\times \cdots\times I^d$ and $J = J^1\times \cdots\times J^d$.
	 Moreover for $s=1,\ldots,d$ and an integer $0\leq i\leq r_s$, let 
	$B_{i}^s(x) = B_{i}^s(x;I^s) = (x-\inf I^s)^i (\sup I^s - x)^{r_s-i}$ be the Bernstein 
	polynomials of degree $r_s$ on $I^s$.

	Then, there exists a constant $C_{\vc{r}}$ not depending on $R$
	so that for each rectangle 
	$ K \in \{ I_1,\ldots,I_{2d}\}$,
	we have
	\begin{equation}\label{eq:better_than_CS}
		\sum_{\vc{0}\leq\vc{ m}\leq \vc{r}} |a_\vc{m}|^2\| B_\vc{m}\charfun_K\|_2^2 \leq C_{\vc{r}} \Big\| \sum_{\vc{0}\leq \vc{m}\leq \vc{r}}a_\vc{m }B_\vc{m}\charfun_K\Big\|_2^2
	\end{equation}
	for each choice of coefficients $(a_\vc{m})$, where we use the notation $B_\vc{m} = B_{m_1}^1\otimes \cdots \otimes B_{m_d}^d$
	and $\vc{0}\leq \vc{m} \leq \vc{r}$ if and only if $0\leq m_i\leq r_i$ for each $i=1,\ldots,d$.
\end{prop}
\begin{proof}
	First, consider $d=1$ and without restriction, let $I=[0,1]$ and $ K = [0,\delta]$ for some $\delta>0$.	
	The case $K=[1-\delta,1]$ then follows by the symmetry of the Bernstein polynomials.
	We use a renormalization of the Bernstein polynomials given by
	$B_{n,\delta} = B_n / \delta^{n+1/2}$. Then we have to show
	\begin{equation}\label{eq:renormalized}
		\sum_{m=0}^r |b_m|^2 \|B_{m,\delta}\charfun_K\|_{2}^2 
		\leq C_r \Big\| \sum_{m=0}^r b_m B_{m,\delta}\charfun_K\Big\|_2^2
	\end{equation}
	for each choice of coefficients $(b_m)$. By homogeneity of both sides, it suffices 
	to consider $ b = (b_m)\in S^r = \{ x\in \mathbb R^{r+1} : |x| = 1\}$.
	Expanding the right hand side, we can equivalently show that 
	\[
		\langle\Lambda_\delta b,b\rangle \leq C_r \langle G_{\delta} b,b\rangle
	\]	
	with the matrices $\Lambda_\delta = \operatorname{diag}(\|B_{m,\delta}\charfun_K\|_2^2)_{m=0}^r$,
	$ G_\delta = (\langle B_{m,\delta}\charfun_K,B_{n,\delta}\charfun_K\rangle)_{n,m=0}^r$ and
	the notation $b = (b_0,\ldots,b_r)$.
	Observe that
	\[
		\lim_{\delta\to 0} \langle B_{m,\delta}\charfun_K, B_{n,\delta}\charfun_K\rangle		
		= \lim_{\delta\to 0} \frac{1}{\delta^{m+n+1}} \int_0^\delta x^{m+n}(1-x)^{2r-m-n}\dif x
		= \frac{1}{m+n+1}.
	\]
	Therefore, defining the matrices $\Lambda = \operatorname{diag}(1/(2m+1))$ and
	$G=(1/(m+n+1))_{m,n=0}^r$ we see that $\lim_{\delta\to 0} \Lambda_\delta = \Lambda$
	and $\lim_{\delta\to 0} G_\delta = G$.
	The matrices $G$ (which is a Hilbert matrix) and $\Lambda$ are symmetric, invertible and positive definite.
	Then, we choose $\delta_0$ such that for all $\delta\leq \delta_0$, we have the two estimates
	\begin{equation}\label{eq:approx}
		\|G_\delta - G \|_2 \leq \frac{\lambda_{\rm min}(G)}{2},\qquad  
		\|\Lambda_\delta - \Lambda\|_2 \leq \frac{\lambda_{\rm min}(\Lambda)}{2},
	\end{equation}
	where we write $\lambda_{\rm min}(A)$ for the smallest eigenvalue of 
	the matrix $A$.

	Denote by $f : S^r\times [\delta_0,1]\to [0,\infty)$ the continuous function 
	mapping $(b_0,\ldots,b_r,\delta)$ with $\sum_{m=0}^r |b_m|^2=1$ to 
	the quotient of the right hand side of \eqref{eq:renormalized} and the
	left hand side of \eqref{eq:renormalized} without the constant $C_r$, given by 
	\[
		f(b_0,\ldots,b_r,\delta) :=	 \frac{\big\| \sum_{m=0}^r b_m B_{m,\delta}\charfun_K\big\|_2^2}{\sum_{m=0}^r |b_m|^2 \| B_{m,\delta}\charfun_K\|_2^2}.
	\]
	Since for each fixed $\delta$, the Bernstein polynomials $(B_{n,\delta})_{n=0}^r$
	are linearly independent on the interval $[0,\delta]$, we have that for each choice of parameters,
	$f(b_0,\ldots,b_r,\delta) >0$. Therefore, 
	by compactness of the domain $S^r \times [\delta_0,1]$,
	$f$ admits a positive minimal value $c$.
	If now $\delta\leq \delta_0$, by the estimates \eqref{eq:approx}, we have for all $b$
	\[
		\frac{\langle \Lambda_\delta b,b\rangle }{\langle \Lambda b,b\rangle}, 	
		\frac{\langle G_\delta b,b\rangle}{\langle Gb,b\rangle}	 \in (1/2,3/2).
	\]	
	Therefore, instead of estimating the quotient $\langle G_\delta b,b\rangle /\langle \Lambda_\delta b,b\rangle$ 
	from below, it suffices to estimate the quotient $\langle Gb,b\rangle /\langle \Lambda b,b\rangle$ from below.
	But this is at least $\lambda_{\rm min}(G) / \lambda_{\rm max}(\Lambda)$, which is positive.
	This shows \eqref{eq:renormalized} in the case $d=1$ for some constant $C_r$.

	If $d>1$ and $I = [0,1]^d$, each rectangle $K = K^1\times \cdots\times K^d$ is 
	such that for each $s = 1,\ldots,d$ we either have
	$K^s = [0,\delta]$ or $[1-\delta,1]$ for some $\delta \in (0,1]$. Since the 
	Bernstein polynomials are tensor products of univariate Bernstein polynomials,
	we can use the derived univariate estimate \eqref{eq:renormalized} repeatedly
	to obtain \eqref{eq:better_than_CS} in general.
\end{proof}

Summarizing Propositions~\ref{prop:almost_orthogonality},~\ref{prop:uniform_parallelity}, we obtain
\begin{co}\label{cor:stab_bernstein}
	For every $1\leq p\leq \infty$ and every tuple  $\vc{r}=(r_1,\ldots,r_d)$ of non-negative integers there exists a constant $C$
	such that for every ring $R$ 
	the Bernstein polynomials $(B_\vc{m})_{\vc{0}\leq \vc{m}\leq \vc{r}}$ of degree $\vc{r}$ rescaled to
	the smallest rectangle $I$ containing $R$
	satisfy the inequalities
	\[
		\Big\| \sum_{\vc{0}\leq \vc{m}\leq \vc{r}}a_{\vc{m}} B_\vc{m } \charfun_R \Big\|_p \leq 
		\sum_{\vc{0}\leq \vc{m}\leq \vc{r}} |a_\vc{m}| \|B_\vc{m}\charfun_R\|_p \leq C 
		\Big\| \sum_{\vc{0}\leq \vc{m}\leq\vc{ r}}a_{\vc{m }}B_{\vc{m}} \charfun_R \Big\|_p.
	\]	
\end{co}

Then, 
Corollary~\ref{cor:stab_bernstein} shows
(2) of Proposition~\ref{prop:stability}
	for the space $S=\mathscr P_\vc{r}$ for arbitrary $d$-tuples of non-negative integers $\vc{r}$,
if we choose 
$S_\vc{m}(R) := \operatorname{span}\{ B_\vc{m}(\cdot;I)\}$ and $u_\vc{m} \propto B_\vc{m}$ for $\vc{0}\leq\vc{ m}\leq \vc{r}$.
Therefore, in order to show condition $\wIIs(\mathscr A,\mathscr P_{\vc r},p,\tau)$, we only
need to investigate condition (1) of Proposition~\ref{prop:stability} for the Bernstein polynomials 
$B_\vc{m}(\cdot;I)$ rescaled to the smallest rectangle $I$ containing the ring $R$.

First, treat the univariate case. A ring $R$ is the disjoint union of two 
intervals $R_-$ and $R_+$ with $R_-$ lying to the left of $R_+$. 
Let $\mathscr X = (X_i)_{i=1}^n$ be a full chain with $R = R(\mathscr X) = X_1\setminus X_n$.
Observe that for $X\in \mathscr X$ with $\pp(X)\in \mathscr X$, we have 
$\bb(X)\subset R_-$ or $\bb(X)\subset R_+$.

\begin{prop}\label{prop:univariate_bernstein}
	Let $r$ be a non-negative integer and fix $\rho \in (0,1)$.

	Then, condition $\wIIs(\mathscr A,\mathscr P_r,p,\tau)$ is equivalent to the existence of a constant $C$
	 such that for each $\rho$-fat full chain $\mathscr X = (X_i)_{i=1}^n$ with $I=X_1$ and 
	each $T\in \{ R_-,R_+\}$ we have
	\begin{equation}\label{eq:bernstein_special}
		\sum_{X\in \mathscr X : \bb(X)\subset T} |\bb(X)|^{\tau/p} 
		\lesssim C^\tau \big(|T|^{\tau/p} + |R\setminus T|^{\tau r+\tau/ p} |I|^{-\tau r}\big),
	\end{equation}
	where we use the notation $R = R(\mathscr X)$ and $R_-,R_+$ for the left and right
	part of $R$, respectively.
\end{prop}
	Before we begin the proof, we note that condition \eqref{eq:bernstein_special}
	is natural in the sense that $\wIIs(\mathscr A,\mathscr P_r,p,\tau)$ yields 
	\eqref{eq:bernstein_special} as a necessary condition after inserting for $f$ the 
	Bernstein polynomials
	 $B_0$ and $B_r$. As it turns out, \eqref{eq:bernstein_special} is also 
	 sufficient for $\wIIs(\mathscr A,\mathscr P_r,p,\tau)$.
\begin{proof}
	We fix $\rho<1$ and a $\rho$-fat full chain $\mathscr X= (X_i)_{i=1}^n$.
	Then, let  $I=X_1=[a,b]$ and we use the notation $B_n = B_n(\cdot; I)$
	for the Bernstein polynomials rescaled to the interval $I$.
	Before we begin the proof, 
	observe that since $\mathscr X$ is $\rho$-fat, we have 
	\[
		|R_-| + |R_+| \leq (1-\rho) |I|.	
	\]	
	Therefore, for each $\ell=0,\ldots,r$,
	\begin{equation}\label{eq:bernstein_equiv}
		\|B_\ell \charfun_R\|_p^p = \int_{R_-\cup R_+} (t-a)^{p\ell}(b-t)^{p(r-\ell)}\dif t
		\simeq |R_-|^{p\ell +1} |I|^{p(r-\ell)} + |R_+|^{p(r-\ell)+1} |I|^{p\ell}.
	\end{equation}	

	We first show the necessity of 
	\eqref{eq:bernstein_special} and first consider $T = R_-$.
	We insert $u_0(t) = B_0(t) = (b-t)^r$ into condition $\wIIs$ to get 
	\begin{equation}\label{eq:w2Bernstein}
		\sum_{X\in\mathscr X : \pp(X)\in\mathscr X}	 \|B_0 \charfun_{\bb(X)}\|_p^\tau 
		\leq M^\tau \|B_0\charfun_R\|_p^\tau
	\end{equation}
	for some constant $M$.
	Assuming  $\bb(X) \subset T = R_-$, 
	we obtain 
	\begin{equation}
		\label{eq:equiv_bernstein}
		\|B_0\charfun_{\bb(X)}\|_p \simeq |I|^r |\bb(X)|^{1/p}.		
	\end{equation}
	Inserting this equivalence and \eqref{eq:bernstein_equiv} for $\ell = 0$ 
	into \eqref{eq:w2Bernstein} we get
	\[
		\sum_{X\in \mathscr X : \bb(X)\subset T}	|I|^{\tau r} |\bb(X)|^{\tau/p} \lesssim
		M^\tau \big(|R_-|^{\tau/p} |I|^{\tau r} + |R_+|^{\tau r + \tau/p}\big).
	\]
	Dividing by $|I|^{\tau r}$ yields the conclusion \eqref{eq:bernstein_special}
	for $T = R_-$. The proof for $T=R_+$ is similar, but uses $u_r = B_r$ instead of $u_0 = B_0$.

	Now we proceed by showing the sufficiency of \eqref{eq:bernstein_special} 
	for $\wIIs$ by showing that \eqref{eq:bernstein_special} implies
	(1) of Proposition~\ref{prop:stability} for Bernstein polynomials. Then, as we already established
	(2) of Proposition~\ref{prop:stability} for Bernstein polynomials above in Corollary~\ref{cor:stab_bernstein}, we use Proposition~\ref{prop:stability}
	to get $\wIIs$.

	Since for each $X\in \mathscr X$ with $\pp(X) \in \mathscr X$ we either have $\bb(X)\subset R_-$
	or $\bb(X)\subset R_+$, it is enough to show, for each $\ell = 0,\ldots,r$ and each 
	$T\in\{R_-,R_+\}$ the inequality
	\begin{equation}\label{eq:Bl}
		\sum_{X\in\mathscr X : \bb(X)\subset T} \| B_\ell \charfun_{\bb(X)}\|_p^\tau	\leq 
		M^\tau \|B_\ell\charfun_R\|_p^\tau
	\end{equation}
	with some constant $M$. We consider again only the case $T=R_-$, since the proof for $T=R_+$ proceeds
	similarly. For $\ell = 0$, as shown above, \eqref{eq:Bl} is just assumption \eqref{eq:bernstein_special}.
	If $\ell = 1,\ldots,r$, observe that $B_\ell(t) \lesssim |R_-|^\ell |I|^{r-\ell}$ on $R_-$.
	Therefore,
	\begin{align*}
		\Big(\sum_{X\in\mathscr X : \bb(X)\subset R_-} \| B_\ell \charfun_{\bb(X)}\|_p^\tau\Big)^{1/\tau}
		&\lesssim |R_-|^{\ell} |I|^{r-\ell} \Big( \sum_{X\in\mathscr X: \bb(X)\subset R_-}|\bb(X)|^{\tau/p} \Big)^{1/\tau}.
	\end{align*}
	Estimate the latter expression by using \eqref{eq:bernstein_special} to get
	\begin{align*}
	\Big(\sum_{X\in\mathscr X : \bb(X)\subset R_-} \| B_\ell \charfun_{\bb(X)}\|_p^\tau\Big)^{1/\tau}
		&\lesssim C |R_-|^\ell |I|^{r-\ell}( |R_-|^{1/p}  + |R_+|^{r+1/p} |I|^{-r}) \\
		&\lesssim C (|R_-|^{\ell+1/p}|I|^{r-\ell} + |R_+|^{r-\ell+1/p} |I|^\ell) \\
		&\simeq C \|B_\ell\charfun_R\|_p,
	\end{align*} 
	where in the latter equivalence, we used \eqref{eq:bernstein_equiv}.
\end{proof}
\begin{example}\label{ex:5.5}
	Let $r$ be a non-negative integer.
	For fixed parameters $1<p<\infty$ and $0<\tau<p$ and 
	using condition~\eqref{eq:bernstein_special} of Proposition~\ref{prop:univariate_bernstein},
	we now give explicit examples of atoms $\mathscr A$ so that 
	$\wIIs(\mathscr A,\mathscr P_r,p,\tau)$ is true, but condition
	$\wIIs(\mathscr A,\mathscr P_{r+1},p,\tau)$ is not true. 
	
	Let $n$ be an arbitrary positive integer and let $\rho\in (0,1)$.
	Moreover, choose $1 < p <\infty$ and $0 < \tau <p$ and let $r\geq 0$ be an integer.
	We choose the constant $c>0$ such that 
	\[
	c \sum_{j=0}^\infty 2^{-j} \leq \frac{1-\rho}{2}
	\]
	and assume that $n$ is sufficiently large so that $n 2^{-\omega n} \leq (1-\rho)/2$ 
	with $\omega := pr+1>0$.
	Then, we define the $\rho$-fat full chain $\mathscr X = (X_j)_{j=-1}^{2n}$ by the formulas
	\begin{align*}
		X_{2\ell -1} = \Big[\ell 2^{-\omega n}, 1 - c \sum_{s=0}^{\ell-1} 2^{-s}\Big],\qquad
		X_{2\ell} = \Big[\ell 2^{-\omega n}, 1 - c \sum_{s=0}^\ell 2^{-s}  \Big],\qquad \ell = 0,\ldots,n.
	\end{align*}
	 Define $\delta := \tau r + \tau/p >0$ and observe that $\delta 
	= \omega\tau /p$. We consider the subchains $\mathscr X' = (X_j)_{j= j_0}^{j_1}$ with 
	$-1 \leq j_0 < j_1 \leq 2n$.
	We now show that the chain $\mathscr X'$ satisfies \eqref{eq:bernstein_special}
	with the parameters $p,\tau,r$ uniformly in $n,j_0,j_1$. On the other hand, we 
	will show that \eqref{eq:bernstein_special} is not satisfied for the parameters $p,\tau,r+1$
	uniformly in $n,j_0,j_1$.
	
	We begin with the first assertion. Let $R = R(\mathscr X')$ and decompose $R = R_-\cup R_+$
	into a left and right part $R_-$ and $R_+$, respectively. Choosing $T=R_+$, we see 
	that \eqref{eq:bernstein_special} is satisfied due to the fact that we have 
	a geometric series on the left hand side of \eqref{eq:bernstein_special}.
	On the other hand, for $T = R_-$, we note that for $X\in\mathscr X'$ with $\bb(X)\subset R_-$,
	we have $|\bb(X)| = 2^{-\omega n}$, which gives 
	\begin{equation}\label{eq:ex_bernstein_special}
	\sum_{X\in\mathscr X' : \bb(X)\subset R_-} | \bb(X)|^{\tau/p} \simeq \frac{j_1 - j_0}{2} 2^{(-\omega n)\tau/p} = 
	\frac{j_1 - j_0}{2} 2^{-\delta n}.
	\end{equation}
	Additionally,
	\[
		|R_+| \simeq c \sum_{s = j_0/2}^{j_1/2} 2^{-s} \simeq c 2^{-j_0/2}.
	\]
	This gives the inequality 
	\[
	\sum_{X\in\mathscr X' : \bb(X)\subset R_-} |\bb(X)|^{\tau/p} \lesssim |R_+|^\delta	
	\]
	with some implicit constant that is uniform in $n,j_0,j_1$, which 
	directly implies \eqref{eq:bernstein_special} for the parameters $p,\tau,r$.
	
	Next, consider the parameters $p,\tau,r+1$. Choose the special subchain $\mathscr X' 
	= (X_j)_{j=j_0}^{j_1}$ with the parameters $j_0 = 2(n-\log n)$ and $j_1 = 2n$.
	Formula \eqref{eq:ex_bernstein_special} for this setting yields
	\begin{equation}\label{eq:lhs_special}
		\sum_{X\in\mathscr X' : \bb(X)\subset R_-} | \bb(X)|^{\tau/p} \simeq (\log n)2^{-\delta n}.
	\end{equation}
	The right hand side of \eqref{eq:bernstein_special} becomes 
	\begin{equation}\label{eq:rhs_special}
	|R_-|^{\tau/p} + |R_+|^{\delta+\tau} \simeq (\log n)^{\tau/p} 2^{-\delta n} + (c 2^{- (n-\log n)})^{\delta + \tau}
	 \simeq  \big( (\log n)^{\tau/p}  + 2^{-n\tau} n^{\delta + \tau}\big) 2^{-\delta n}.
	\end{equation}
	Since $0 < \tau < p$, comparing \eqref{eq:lhs_special} and \eqref{eq:rhs_special} 
	yields that inequality \eqref{eq:bernstein_special} for the parameters $p,\tau,r+1$ 
	is impossible uniformly in $j_0,j_1,n$. This 
	finishes the proofs of the above claims.

	We construct atoms $\mathscr A$ that contain, for 
	each positive integer $n$, a rescaling of the $\rho$-fat full chain 
	given above. We assume that $\mathscr A$ does not contain more $\rho$-fat full chains.
	Since \eqref{eq:bernstein_special} is invariant under rescaling, we 
	infer from Proposition~\ref{prop:univariate_bernstein} that $\wIIs(\mathscr A,\mathscr P_r,p,\tau)$ is satisfied but 
	$\wIIs(\mathscr A,\mathscr P_{r+1},p,\tau)$ is not satisfied.
	\end{example}

Next, we describe a multivariate analog of the Proposition~\ref{prop:univariate_bernstein}.
Let $\mathscr X = (X_i)_{i=1}^n$ be a $\rho$-fat full chain  and denote $I= X_1 = I^1\times \cdots\times I^d$,
$J=X_n = J^1\times \cdots \times J^d$. Moreover, write $R^s = I^s\setminus J^s$ and
decompose $R^s$ into the union of two disjoint subintervals $R^s_-$ and $R^s_+$ with 
the former lying to the left of the latter. Let
$K_s^{\pm} = I^1\times \cdots\times I^{s-1}\times R^s_{\pm}\times I^{s+1}\times \cdots \times I^d$
and $K_s = K_s^- \cup K_s^+$.
 Then, the multivariate analog of \eqref{eq:bernstein_equiv}
is, for each $\vc{j}=(j_1,\ldots,j_d)$ with $0\leq j_i\leq r_i$,
\[
	\|B_\vc{j} \charfun_{K_s}\|_p^p \simeq 	\Big( \prod_{i\neq s} |I^i|^{pr_i + 1} \Big)
	\big( |R^s_-|^{pj_s + 1} |I^s|^{p(r_s - j_s)} + |R_+^s|^{p(r_s - j_s) + 1} |I^s|^{pj_s} \big).
\]
Moreover, the analog of \eqref{eq:equiv_bernstein} is that if $\bb(X) \subset K_s^-$ and
$\rho$ is sufficiently large, we have for all $\vc{j}=(j_1,\ldots,j_d)$ with $j_s = 0$
\[
	\| B_\vc{j} \charfun_{\bb(X)}\|_p^p\simeq 	 \Big( \prod_{i\neq s} |I^i|^{pr_i + 1} \Big)
	|I^s|^{pr_s} |\bb(X)^s|.
\]
Using those equivalences and the observation $\|B_\vc{j}\charfun_R\|_p \simeq \sum_{s=1}^d \|B_\vc{j}\charfun_{K_s}\|_p$,
we obtain, similarly to the univariate result (Proposition~\ref{prop:univariate_bernstein}), the following multivariate version.
\begin{prop}\label{prop:multivariate_bernstein}
	Let $\vc{r}=(r_1,\ldots,r_d)$ be a $d$-tuple of non-negative integers and fix $\rho \in (0,1)$.

	Then, condition $\wIIs(\mathscr A,\mathscr P_{\vc r},p,\tau)$ is equivalent to the existence of 
	a constant $C$  such that for each $\rho$-fat full chain $\mathscr X=(X_i)_{i=1}^n$
	with $I=X_1$ and $R = X_1\setminus X_n$
	is equivalent to the condition that for each $s=1,\ldots,d$ and each $T\in \{R^s_-, R^s_+\}$ we have,
	denoting $K_s^T = I^1\times \cdots \times I^{s-1}\times T\times I^{s+1}\times
	\cdots \times I^d$,
	\begin{equation}\label{eq:bernstein_special_multi}
		\sum_{X\in \mathscr X : \bb(X)\subset K_s^T} |\bb(X)^s|^{\tau/p} 
		\lesssim C^\tau \big(|T|^{\tau/p} + |R^s\setminus T|^{\tau r_s+\tau/ p} |I^s|^{-\tau r_s} + \sum_{\ell\neq s} |I^s|^{\tau/p}\gamma_{\ell}\big),
	\end{equation}
	where $\gamma_{\ell}$ is given by 
	\[
		\gamma_{\ell}^{1/\tau} = \min_{0\leq j_\ell\leq r_\ell} |I^\ell|^{-1/p}
		\big( |R^\ell_-|^{j_\ell+1/p} |I^\ell|^{-j_\ell} + |R^\ell_+|^{r_\ell-j_\ell + 1/p} |I^\ell|^{j_\ell-r_\ell}\big).
	\]		
\end{prop}

\begin{rem}
(1) With the notation used in the formulation 
of Proposition~\ref{prop:multivariate_bernstein}, we remark that for each $X\in \mathscr X$,
the set $\bb(X)$ is contained in precisely one of the sets  $K_s^T$ for $s \in\{1,\ldots,d\}$ and $T\in\{R_-^s,R_+^s\}$.
This is due to the fact that for $X_n \subset X = [a_1,b_1]\times \cdots\times [a_d,b_d]$ 
there exists precisely one direction $s\in \{1,\ldots,d\}$ such that 
$\bb(X) =  [a_1,b_1]\times \cdots \times [a_{s-1},b_{s-1}] \times U \times [a_{s+1},b_{s+1}]\times \cdots \times [a_d,b_d]$
for some interval $U$ not intersecting $[a_s,b_s]$. This situation corresponds 
to the fact that $\bb(X)\subset K_s^T$ for this index $s$ and $T = R^s_-$ if 
$U$ is to the left of $[a_s,b_s]$ and $T=R^s_+$ if $T$ is to the right of $[a_s,b_s]$. \\
(2) As in the univariate case of Proposition~\ref{prop:univariate_bernstein}, condition 
\eqref{eq:bernstein_special_multi} is natural in the sense that 
$\wIIs(\mathscr A,\mathscr P_{\vc{r}},p,\tau)$ yields \eqref{eq:bernstein_special_multi} as 
a necessary condition after inserting Bernstein polynomials $B_\vc{i}$ for $\vc{i}=(i_1,\ldots,i_d)$ such that $i_j \in \{0, r_j\}$ for some $j=1,\ldots, d$.
Similarly, we then also see that \eqref{eq:bernstein_special_multi} is also 
sufficient for $\wIIs(\mathscr A,\mathscr P_\vc{r},p,\tau)$.
\end{rem}

Using Proposition~\ref{prop:multivariate_bernstein} and Example~\ref{ex:5.5}, we next  construct two examples of 
atoms $\mathscr A$ and $\mathscr A'$ and two polynomial spaces $S$ and $S'$ such that 
$\wIIs(\mathscr A,S,p,\tau)$ is satisfied but $\wIIs(\mathscr A,S',p,\tau)$ is not 
and  $\wIIs(\mathscr A',S',p,\tau)$ is satisfied but $\wIIs(\mathscr A',S,p,\tau)$ is not.
This shows that the condition $\wIIs$ indeed depends on the space $S$.

\begin{example}\label{ex:depS}
	Fix the parameters $1<p<\infty$ and $0<\tau<p$. Fix the non-negative integer 
	$\kappa$ and the coordinate $i \in \{1,\ldots,d\}$.
	We want to construct a partition $\mathscr A = \mathscr A(\kappa,i)$ of the unit cube $[0,1]^d$,
	 that depends on 
	the values of $\kappa$ and $i$ such that condition 
	$\wIIs(\mathscr A, \mathscr P_\vc{r}, p,\tau)$ is satisfied for $\vc{r}=(r_1,\ldots,r_d)$ 
	if and only if $r_i \leq \kappa$. 
	We construct a
	binary filtration $(\mathscr F_n)_{n=0}^\infty$ such that $\mathscr F_{n+1}$ 
	is obtained by dividing one atom of $\mathscr F_n$ (which is a rectangle) into 
	two rectangles. First, observe that an atom $A = A^1\times \cdots \times A^d$
	is split into two disjoint atoms $B = B^1\times \cdots \times B^d$ and 
	$C = C^1\times \cdots \times C^d$ such that there is exactly one direction $j$ 
	such that $B^j \cap C^j= \emptyset$ and  $A^j = B^j \cup C^j$ with $|B^j|, |C^j| > 0$ and, for every $\ell \neq j$
	we have $A^\ell = B^\ell = C^\ell$. In this case, we say that $A$ splits into 
	$B\cup C$ in direction $j$.
	We now assume $\mathscr A$ is such that if atoms $A$ are split into $B\cup C$ 
	in direction $j\neq i$, we have that $|B| = |C| = |A| / 2$ and the splitting 
	in direction $i$ contains chains as in Example~\ref{ex:5.5} of arbitrary 
	length $n$ for the parameters 
	$p,\tau$ and $r = \kappa$.
	In this construction, the only $\rho$-fat  chains ($\rho>1/2$) contained in $\mathscr A$ are 
	the ones from Example~\ref{ex:5.5} in direction $i$. This is the case
	since every chain $(X_0,X_1)$ of length $2$ with $|X_0| = 2|X_1|$ is already not 
	$\rho$-fat.
	We use the criterion for $\wIIs(\mathscr A, \mathscr P_\vc{r},p,\tau)$  from 
	Proposition~\ref{prop:multivariate_bernstein} given in \eqref{eq:bernstein_special_multi}.
	Here, we get that the 
	needed parameters $\gamma_\ell$ vanish for all $\ell\neq i$ since for those $\rho$-fat
	chains we have $R_-^\ell = R_+^\ell = \emptyset$. Therefore, condition 
	\eqref{eq:bernstein_special_multi} coincides with condition \eqref{eq:bernstein_special}
	of Proposition~\ref{prop:univariate_bernstein}.
	We have shown in Example~\ref{ex:5.5} that this condition is satisfied for 
	$\mathscr P_\vc{r}$ with $\vc{r} = (r_1,\ldots,r_d)$ in the case $r_i \leq  \kappa$, but 
	not satisfied in the case $r_i > \kappa$.
Therefore, we have constructed the desired partition $\mathscr A(\kappa,i)$.

	Now, let $\vc{r} = (r_1,\ldots,r_d)$ and $\vc{r}' = (r_1',\ldots,r_d')$ be $d$-tuples of 
	non-negative integers that are non-comparable, i.e., there exist two different
	indices $i_1,i_2$ such that $r_{i_1} < r_{i_1}'$ and $r_{i_2} > r_{i_2}'$.
	Define $\mathscr A = \mathscr A(r_{i_1},i_1)$ and $\mathscr A' = \mathscr A(r_{i_2},i_2)$.
	Then, $\wIIs(\mathscr A,\mathscr P_\vc{r},p,\tau)$ is satisfied but $\wIIs(\mathscr A,\mathscr P_{\vc{r}'},p,\tau)$
	is not. On the other hand, $\wIIs(\mathscr A',\mathscr P_{\vc{r}'},p,\tau)$ is satisfied 
	but $\wIIs(\mathscr A',\mathscr P_\vc {r},p,\tau)$ is not, showing the desired 
	dependence of $\wIIs$ on its first two parameters.
\end{example}

\subsection{Coordinate affinely invariant polynomial spaces} The aim of this section is to get an analogue of Propositions \ref{prop:univariate_bernstein} and  \ref{prop:multivariate_bernstein} for $\mathscr P_{r,d}$ -- the space of polynomials of degree $r$ in $d$ variables. Note that $\mathscr P_{r,d} = \Span \{ \mathscr P_{\vc{r}}: | \vc{r} | = r $\}. It follows that if 
$\wIIs(\mathscr A,  \mathscr P_{r,d}  ,p,\tau)$ is satisfied, then $\wIIs(\mathscr A,  \mathscr P_{\vc{r}}  ,p,\tau)$  is satisfied for each $\vc{r}$ with $| \vc{r}| =r$.  We are going to prove that the converse is also true. In fact, it is possible to  consider this question in a slightly more general setting.

Let $M = \{ \vc{m}^{(1)},\ldots, \vc{m}^{(s)}\}$ be a finite set consisting of 
$d$-tuples of non-negative integers $\vc{m}^{(i)}$. 
Denote $\mathscr P_{M} = \Span\{ \mathscr P_{\vc{m}^{(1)}}, \ldots, \mathscr P_{\vc{m}^{(s)}} \}$.
Note that the spaces $\mathscr P_M$ are precisely the linear spaces of 
polynomials that are invariant under coordinate affine transformations.
Let the atoms $\mathscr A$ consist of subintervals (rectangles) of $[0,1]^d$ and $\mathbb P$ be the
$d$-dimensional  Lebesgue measure. 

In this setting, we prove the following:

\begin{theo}
\label{theo.span}
	Let $M = \{ \vc{m}^{(1)},\ldots,\vc{m}^{(s)}\}$ 
and $\mathscr P_{M} = \Span\{ \mathscr P_{\vc{m}^{(1)}}, \ldots, \mathscr P_{\vc{m}^{(s)}} \}$. 
Fix $1 < p < \infty$ and $0<\tau <p$. 

	Then, condition $\wIIs(\mathscr A,\mathscr P_M,p,\tau)$ is satisfied 
	if and only if condition $\wIIs(\mathscr A,\mathscr P_\vc{m},p,\tau)$ is satisfied for each $ \vc{m} \in M $.
\end{theo}

 Combining Theorem \ref{theo.span} with Proposition \ref{prop:multivariate_bernstein} and Theorem \ref{thm:main} we get the geometric characterization of Bernstein inequality $\operatorname{BI}(\mathscr A, \mathscr P_M,p,\tau)$:
\begin{co}
\label{co.geometric}
Let $M = \{ \vc{m}^{(1)},\ldots,\vc{m}^{(s)}\}$ 
and $\mathscr P_{M} = \Span\{ \mathscr P_{\vc{m}^{(1)}}, \ldots, \mathscr P_{\vc{m}^{(s)}} \}$, $1 < p < \infty$ and $0<\tau <p$. 
Then Bernstein inequality $\operatorname{BI}(\mathscr A, \mathscr P_M,p,\tau)$ is satisfied if and only if condition \eqref{eq:bernstein_special_multi} of Proposition \ref{prop:multivariate_bernstein} is satisfied for each $\vc{m} \in M$.
\end{co}

Concerning the proof of Theorem \ref{theo.span}, it is clear that if condition $\wIIs(\mathscr A,\mathscr P_M,p,\tau)$
is satisfied, then   $\wIIs(\mathscr A,\mathscr P_\vc{m},p,\tau)$ is satisfied for each $ \vc{m} \in M $.
Our proof of the converse implication relies on  Proposition~\ref{prop:stability_weaker}.
Looking at Proposition~\ref{prop:stability_weaker}, we see that the only thing
missing here is some inequality similar to the one in Corollary~\ref{cor:stab_bernstein},
but for a certain decomposition of $u\in\mathscr P_M$ into $u_1 + \cdots + u_s$ with
$u_j\in \mathscr P_{\vc{m}^{(j)}}$. The construction of uch a decomposition
is the goal of this section, and it relies on the analytical Proposition \ref{lem.hermite.multi}, 
and the combinatorial  Proposition \ref{prop:decomp}.

We begin with the analytical part of the proof. First, we  give some univariate ingredients.
Fix $m_0,m_1 \geq 0$ and consider the interpolation operator $H_{m_0,m_1}$
that maps sufficiently regular functions $f$ on $[0,1]$ to the unique
polynomial $g$ of degree $m_0 + m_1 - 1$ on $[0,1]$ satisfying the conditions
\begin{align*}
	g^{(k_0)}(0) &= f^{(k_0)}(0),\qquad k_0 = 0,\ldots,m_0-1, \\
	g^{(k_1)}(1) &= f^{(k_1)}(1),\qquad k_1 = 0,\ldots,m_1-1.
\end{align*}
For $m\geq 1$, denote by 
\[
	\mathscr H_m = \{ H_{j,m-j} : j=0,\ldots,m\}	
\]
the set of the projections defined above whose  
range is  $\mathscr P_{m-1}$. In particular,
each operator $H_{j,m-j}$ can be viewed as a projection from $\mathscr P_m$ to 
$\mathscr P_{m-1}$.

\begin{lem}\label{lem:hermite}
	Fix $1\leq p\leq \infty$ and a non-negative integer $m$. Then, there exists a 
	constant $C$ such that for each ring $R = R_- \cup R_+$ we can choose $H\in \mathscr H_m$
	such that 
	\[
		\| H : (\mathscr P_m, L^p(R)) \to (\mathscr P_{m-1},L^p(R)) \| \leq C.
	\]	

	More precisely, this choice is  $H = H_{j,m-j}$ with $j$ satisfying
	\[
		|R_-|^{j+1/p} + |R_+|^{m-j+1/p} = \min_{0\leq i\leq m}	\big(|R_-|^{i+1/p} + |R_+|^{m-i+1/p}\big).
	\]
\end{lem}
\begin{proof}
	Let $m_0,m_1$ be non-negative integers with $m=m_0 + m_1$.
	In the course of the proof, we use the functions $u_{\ell_0,0}, u_{\ell_1,1}\in \mathscr P_{m-1}$
	for $\ell_0 \in \{0,\ldots, m_0 - 1\}$ and $\ell_1\in \{0,\ldots,m_1-1\}$	that
	are given by the conditions
	\begin{align*}
		u_{\ell_0,0}^{(\ell_0)}(0) &= 1, \\
		u_{\ell_0,0}^{(k_0)}(0) &= u_{\ell_0,0}^{(k_1)}(1)=0,\qquad \text{for }
		0\leq k_0\neq \ell_0\leq m_0 -1\text{ and } 0\leq k_1\leq m_1-1
	\end{align*}
	and 
	\begin{align*}
		u_{\ell_1,1}^{(\ell_1)}(1) &= 1, \\
		u_{\ell_1,1}^{(k_0)}(0) &= u_{\ell_1,1}^{(k_1)}(1)=0,\qquad \text{for }
		0\leq k_0\leq m_0 -1\text{ and } 0\leq k_1\neq \ell_1\leq m_1-1.
	\end{align*}
	Observe that 
	\begin{align*}
		u_{\ell_0,0}(x) = x^{\ell_0} (1-x)^{m_1} v_{\ell_0,0}(x),\qquad\text{and}\qquad
		u_{\ell_1,1}(x) =  x^{m_0} (1-x)^{\ell_1} v_{\ell_1,1}(x)
	\end{align*}
	for some polynomials $v_{\ell_0,0}\in \mathscr P_{m_0-1-\ell_0}$ and 
	$v_{\ell_1,1}\in \mathscr P_{m_1-1-\ell_1}$. Since the polynomials $u_{\ell_0,0}, u_{\ell_1,1}$
	as well as the polynomials $v_{\ell_0,0},v_{\ell_1,1}$ do not depend on the 
	choice of the ring $R = R_-\cup R_+$ with $R_- = [0,\alpha]$ and $R_+=[1-\beta,1]$ for some numbers $\alpha,\beta$,
	 we obtain for all possible values of $\ell_0,\ell_1$
	\begin{equation}\label{eq:hermite_u}
		\begin{aligned}
			\|u_{\ell_0,0}\|_{L^p(R)} &\lesssim |R_-|^{\ell_0+1/p} + |R_+|^{m_1 + 1/p} = \alpha^{\ell_0 + 1/p} + \beta^{m_1 + 1/p}, \\
			\|u_{\ell_1,1}\|_{L^p(R)} &\lesssim |R_-|^{m_0 + 1/p} + |R_+|^{\ell_1 + 1/p} = \alpha^{m_0 + 1/p} + \beta^{\ell_1 + 1/p}. 
		\end{aligned}
	\end{equation}

	Now, let $w\in \mathscr P_m$ and estimate the projection 
	$H_{m_0,m_1}w = \sum_{\ell_0=0}^{m_0-1} w^{(\ell_0)}(0) u_{\ell_0,0}
					+ \sum_{\ell_1=0}^{m_1-1} w^{(\ell_1)}(1) u_{\ell_1,1}$ 
	by \eqref{eq:hermite_u} as follows:
	\begin{equation}\label{eq:hermite}
		\begin{aligned}
				\| &H_{m_0,m_1} w\|_{L^p(R)}	\\
				&\lesssim 
				\sum_{\ell_0=0}^{m_0-1} |w^{(\ell_0)}(0)| (\alpha^{\ell_0+1/p} + \beta^{m_1+1/p})
				+ \sum_{\ell_1=0}^{m_1-1} |w^{(\ell_1)}(1)| (\alpha^{m_0+1/p} + \beta^{\ell_1 + 1/p}).
		\end{aligned}
	\end{equation}
	Next, we want to estimate $|w^{(\ell_0)}(0)|$ and $|w^{(\ell_1)}(1)|$ in 
	terms of the coefficients of $w$ with respect to the
	Bernstein polynomials $B_i(x) = x^i(1-x)^{m-i}$ of degree $m$. Observe that
	if $w = \sum_{i=0}^m a_i B_i$, we have 
	for $k = 0,\ldots,m$
	\begin{align*}
		w^{(k)}(0) &= \sum_{i=0}^m a_i B_i^{(k)}(0) =  	\sum_{i=0}^k a_i B_i^{(k)}(0), \\
		w^{(k)}(1) &= \sum_{i=0}^m a_i B_i^{(k)}(1) =  	\sum_{i=m-k}^m a_i B_i^{(k)}(1). \\
	\end{align*}
	Thus, there exists a constant $C$ depending only on $m$ such that for all $k = 0,\ldots,m$,
	\[
		|w^{(k)}(0)|\leq C \sum_{i=0}^k |a_i|,\qquad \text{and}\qquad 
		|w^{(k)}(1)| \leq C	\sum_{i=m-k}^m |a_i|.
	\]
	Inserting those estimates into \eqref{eq:hermite} and exchanging the sums gives
	\begin{align*}
		\|H_{m_0,m_1} w\|_{L^p(R)} \lesssim 
		&\sum_{i=0}^{m_0-1} |a_i|\sum_{\ell_0=i}^{m_0-1} (\alpha^{\ell_0+1/p} + \beta^{m_1+1/p}) \\
		&+ \sum_{i=m_0+1}^{m}|a_i| \sum_{\ell_1 = m - i}^{m_1-1} (\alpha^{m_0 + 1/p} + \beta^{\ell_1+1/p}).
	\end{align*}
	Since $\alpha,\beta\leq 1$, we further obtain
	\begin{equation}\label{eq:Hw}
		\|H_{m_0,m_1} w\|_{L^p(R)} \lesssim \sum_{i=0}^{m_0-1}|a_i|(\alpha^{i+1/p} + \beta^{m_1+1/p})	
		+ \sum_{i=m_0+1}^{m_0 + m_1} |a_i| ( \alpha^{m_0 + 1/p} + \beta^{m-i+1/p}).
	\end{equation}

	By Corollary~\ref{cor:stab_bernstein} for dimension $d=1$ and evaluating
	the $p$-norm on $R$ for the functions $B_i$, 
	\begin{equation}\label{eq:w}
		\|w\|_{L^p(R)} \simeq \sum_{i=0}^m |a_i| (\alpha^{i+1/p} + \beta^{m-i+1/p}).		
	\end{equation}
	We want to 
	choose $m_0,m_1$ with $m_0 + m_1 = m$ such that the 
	right hand side of \eqref{eq:Hw} can be estimated from above by 
	the right hand side of \eqref{eq:w}.
	This can be achieved if $m_0,m_1$ are such that
	\[
		\alpha^{m_0 + 1/p}	 + \beta^{m_1+1/p} = \min_{0\leq i\leq m} \big(\alpha^{i+1/p} + \beta^{m-i+1/p}\big).
	\]	
	Indeed, this implies
	\begin{align*}
		\beta^{m_1 + 1/p} &\leq \alpha^{i+1/p} + \beta^{m-i+1/p},\qquad 0\leq i\leq m_0-1, \\
		\alpha^{m_0 + 1/p} &\leq \alpha^{i+1/p} + \beta^{m-i+1/p},\qquad m_0 + 1\leq i\leq m.
	\end{align*}
	Comparing \eqref{eq:Hw} and \eqref{eq:w}, this gives the desired estimate.
\end{proof}

\begin{rem}\label{rem:Hm}
	Reading the above proof with $\alpha=\beta=1/2$, we also obtain that	
	the operator $H$ is also bounded by some constant from $L^p[0,1]$ to $L^p[0,1]$
	for $1\leq p\leq \infty$. All operators contained in $\mathscr H_m$ in fact have that
	property.
\end{rem}

\begin{lem}\label{lem:hermite_2}
	Fix $1\leq p\leq \infty$ and a positive integer $r$. 
	There exists a constant $C$ such that for each ring $R$, and each
	$m = 0,\ldots,r-1$, there is a linear projection $U_m : \mathscr P_r \to \mathscr P_m$, depending on $R$ and $p$, such that
	\begin{enumerate}
		\item \label{it:bound}$\| U_m : (\mathscr P_r, L^p(R)) \to (\mathscr P_m,L^p(R))\| \leq C$,
		\item \label{it:boundinterval}$\| U_m : (\mathscr P_r, L^p[0,1]) \to (\mathscr P_m,L^p[0,1])\| \leq C$,
		\item \label{it:commute}$ U_m \circ U_k = U_k\circ U_m = U_{\min(k,m)}$ for $0\leq k,m\leq r-1$,
		\item \label{it:id}$U_m = \operatorname{Id}$ on $\mathscr P_m$.
	\end{enumerate}
\end{lem}
\begin{proof}
	For the fixed ring $R = [0,\alpha] \cup [1-\beta,1]$, let 
	$H_\ell : \mathscr P_\ell\to \mathscr P_{\ell-1}$ be the projection given by
	Lemma~\ref{lem:hermite} (for $m=\ell$).
	Define $U_m := H_{m+1}\circ H_{m+2} \circ \cdots\circ H_r$.
	The existence of the constant $C$ satisfying \eqref{it:bound} and 
	\eqref{it:boundinterval} follows then from
	Lemma~\ref{lem:hermite} and Remark~\ref{rem:Hm} as well as item \eqref{it:id}. It remains to check \eqref{it:commute}.
	Let $m\geq k$. Since $U_k f\in \mathscr P_k\subset \mathscr P_m$
	and $U_m = \operatorname{Id}$ on $\mathscr P_m$, we get $U_m ( U_k f) = U_kf$.
	On the other hand,
	\[
		U_k = H_{k+1}\circ \cdots \circ H_m \circ H_{m+1}\circ\cdots\circ H_r 
		= H_{k+1}\circ\cdots\circ H_m\circ U_m.		
	\]
	Therefore, 
	\[
		U_k (U_m f) = H_{k+1}\circ\cdots \circ H_m (U_m\circ U_m f)	
		= H_{k+1}\circ\cdots\circ H_m (U_m f) = U_k f,
	\]
	which finishes the proof of \eqref{it:commute} and thus the proof of the lemma as well.
\end{proof}

To formulate Proposition \ref{lem.hermite.multi}, we need the multivariate versions of the projections
from Lemma \ref{lem:hermite_2}.
For this, fix $1 \leq p \leq \infty$ and $\vc{r} = (r_1, \ldots, r_d)$.
Take a ring $R = [0,1]^d\setminus J$ for some rectangle $J= J^1 \times \ldots \times J^d$.
For each $i = 1, \ldots, d$, let $\{U_{m_i}^{(i)}: 0 \leq m_i \leq r_i-1\}$ be the sequence of projections from 
Lemma~\ref{lem:hermite_2}, corresponding to $r = r_i$,  $R^i = [0,1] \setminus J^i$ and $p$; in addition, let 
$U_{r_i}^{(i)} = \operatorname{Id}$. Now, put
\begin{equation}
\label{eq.projection}
U_\vc{m} = U_{m_1}^{(1)}\otimes\cdots \otimes U_{m_d}^{(d)},
\quad \vc{0} \leq \vc{m} \leq \vc{r}.
\end{equation}
Now, we are ready to formulate the analytical ingredient of the proof of Theorem \ref{theo.span}:
\begin{prop}
\label{lem.hermite.multi} 
Fix $1 \leq p \leq \infty$ and $\vc{r}$.
Let $R = [0,1]^d\setminus J$ for some rectangle $J= J^1 \times \ldots \times J^d$, and let $U_\vc{m}$ be defined by \eqref{eq.projection}.

Then, there is a constant $C= C(\vc{r},p,d)$ such that for each $\vc{0} \leq \vc{m} \leq \vc{r}$ 
\begin{enumerate}
		\item \label{it:bound.m}$\| U_\vc{m} : (\mathscr P_\vc{r}, L^p(R)) \to (\mathscr P_\vc{m},L^p(R))\| \leq C$,
		\item \label{it:boundinterval.m}$\| U_\vc{m} : (\mathscr P_\vc{r}, L^p[0,1]^d) \to (\mathscr P_\vc{m},L^p[0,1])\| \leq C$,
		\item \label{it:commute.m}$ U_\vc{m} \circ U_\vc{k} = U_\vc{k}\circ U_\vc{m} = U_{\min(\vc{k},\vc{m})}$ for $\vc{0} \leq \vc{k}, \vc{m}\leq \vc{r}$, where $\min(\vc{k},\vc{m})$ is the $d$-tuple defined by   taking the minimum coordinate-wise. 
		\item \label{it:id.m}$U_\vc{m} = \operatorname{Id}$ on $\mathscr P_\vc{m}$.
	\end{enumerate}
\end{prop}
\begin{proof} Denote $\Gamma_i = \Lambda_1^{(i)} \times \ldots \times \Lambda_d^{(i)}$, where 
$\Lambda^{(i)}_j = R^i$ for $j =i$ and $\Lambda^{(i)}_j = [0,1]$ for $j \neq i$.
Note that $R = \bigcup_{i=1}^d \Gamma_i$, and
$
\| f \|_{L^p(R)} \simeq \sum_{i=1}^d \| f \|_{L^p(\Gamma_i)},
$
with the implied constants depending on $d$ and $p$, but not on $R$.

With this equivalence at hand,
Proposition \ref{lem.hermite.multi} is a consequence of Lemma \ref{lem:hermite_2}. 
\end{proof}

Now, we turn to the combinatorial ingredient of the proof of Theorem \ref{theo.span}.
For this,  fix $M=\{ \vc{m}^{(1)},\ldots, \vc{m}^{(s)}\}$.  
Given an arbitrary, non-empty subset 
$B \subset\{1,\ldots,s\}$, we denote $\vc{m}_B = \min_{\ell\in B}\vc{ m}^{(\ell)}$,
where we take the minimum coordinate-wise. 
\begin{prop}\label{prop:decomp}
Fix $1 \leq p \leq \infty$ and a ring $R = [0,1]^d \setminus J$.
	Let $ M = \{\vc{ m}^{(1)},\ldots,\vc{m}^{(s)} \}$ be a collection $d$-tuples of non-negative integers,
	and let $\vc{r}$ be such that $ \vc{0} \leq \vc{m}^{(j)} \leq \vc{r}$ for all $j=1, \ldots, s$.
	Let the projections $U_{\vc{m}}$ be given by \eqref{eq.projection} for $\vc{0} \leq \vc{m} \leq \vc{r}$.
	
	Define
	\[
		W_M = \sum_{\emptyset \neq B \subset \{1,\ldots,s\}} (-1)^{\card B + 1} U_{\vc{m}_B},
	\]	
	where we write $\card B$ for the cardinality of the set $B$.

	Then, $W_M = \operatorname{Id}$ on $\mathscr P_M$.
\end{prop}
\begin{proof}
	We prove this statement inductively on $s$. Clearly, it is true for $s=1$.
	If $s\geq 2$, and $M = \{ \vc{m}^{(1)},\ldots,\vc{m}^{(s)}\}$, 
	we define, for $j = 1,\ldots,s$, the sets
	$M_j = M\setminus \{ \vc{m}^{(j)}\}$.
	By the inductive assumption, we know that for each $j=1,\ldots,s$, the operators 
	\begin{equation*}
		W_{M_j}	= \sum_{\emptyset\neq C\subset \{1,\ldots,s\}\setminus \{j\}}	(-1)^{\card C+1} U_{\vc{m}_C}
	\end{equation*}
	satisfy $W_{M_j} = \operatorname{Id}$ on $\mathscr P_{M_j}$.

	In order to prove $W_M = \operatorname{Id}$ on $\mathscr P_M$, it is enough to 
	prove it on $\mathscr P_{M_1},\ldots,\mathscr P_{M_s}$ separately, since the operators $U_{\vc{m}_C}$ 
	are linear. Therefore, fix $j\in \{1,\ldots,s\}$ and write
	\begin{align*}
		W_M &= \sum_{\emptyset \neq C\subset \{1,\ldots,s\}\setminus \{j\} }	
			(-1)^{\card C + 1}	 U_{\vc{m}_C} + U_{\vc{m}^{(j)}} 
			\Big(\operatorname{Id} + \sum_{\emptyset \neq C \subset \{1,\ldots,s\}\setminus \{j\}}
				(-1)^{\card C} U_{\vc{m}_C} \Big) \\
				&= W_{M_j} + U_{\vc{m}^{(j)}}  ( \operatorname{Id} - W_{M_j}),
	\end{align*}
	where we used that the operators $U_{\vc{m}^{(j)}}$ commute among each other.
	By induction hypothesis $W_{M_j} = \operatorname{Id}$ on $\mathscr P_{M_j}$ and therefore,
	we also have $W_M = \operatorname{Id}$ on $\mathscr P_{M_j}$ .
\end{proof}

We are ready to prove Theorem \ref{theo.span}.
\begin{proof}[Proof of Theorem \ref{theo.span}] As observed above, if condition $\wIIs(\mathscr A,\mathscr P_M,p,\tau)$ is satisfied, then all conditions
$\wIIs(\mathscr A,\mathscr P_{\vc{m}^{(j)}},p,\tau)$,  $j = 1, \ldots, d$ are satisfied as well.

We need to check the converse.  Assume that  all conditions $\wIIs(\mathscr A,\mathscr P_{\vc{m}^{(j)}},p,\tau)$,  $j = 1, \ldots, d$
are satisfied.
 Note that this implies that $\wIIs(\mathscr A,\mathscr P_{\vc{m}},p,\tau)$ is satisfied for $\vc{m}$ such that $\vc{0} \leq \vc{m} \leq
\vc{m}^{(j)}$ for some $j = 1, \ldots, d$.
In particular,  this means that for each $\vc{m}_B$, $\emptyset \neq B \subset \{1, \ldots , s\}$,  as defined above,  condition 
$\wIIs(\mathscr A,\mathscr P_{\vc{m}_B},p,\tau)$ is satisfied. 
Therefore, according to Proposition~\ref{prop:stability_weaker}, it is enough to point out  
for each ring $R = I\setminus J$ and each $u \in \mathscr P_M$ a decomposition $u=\sum_{\emptyset \neq B \subset \{1, \dots, s\}} u_B$
with $u_B\in \mathscr P_{\vc{m}_B}$ satisfying
\begin{equation}\label{eq:w2toshow}
\sum_{\emptyset \neq B \subset \{1, \dots, s\}}	 \|u_B\charfun_R\|_p \leq C' \|u\charfun_R\|_p
\end{equation}
for some constant $C'$.

For this, fix $\vc{r}$ such that $ \vc{m}^{(j)} \leq \vc{r}$ for each $j = 1,\ldots , s$. 
Note that the spaces of polynomials which we are dealing with are invariant under coordinate-wise affine change of variables.
Therefore, by a rescaling argument, Propositions \ref{lem.hermite.multi} and \ref{prop:decomp} are extended to all  rings $R = I \setminus J$,  with $I$ replacing $[0,1]^d$  in Proposition \ref{lem.hermite.multi} \eqref{it:boundinterval.m}, with the same constants.

Now, given a ring $R = I\setminus J$, 
 let $U_\vc{m}$ be given by Proposition \ref{lem.hermite.multi} for this $\vc{r}$, $R$ and $p$.
 For $u \in \mathscr P_M$ and $\emptyset \neq B \subset \{1, \dots, s\}$, put $u_B = (-1)^{\card B + 1} U_{\vc{m}_B} u$.
 Clearly, $u_B \in \mathscr P_{\vc{m}_B}$.
 It follows by Propositions \ref{lem.hermite.multi} and \ref{prop:decomp} that $u=\sum_{\emptyset \neq B \subset \{1, \dots, s\}} u_B$ is a  decomposition of $u$ satisfying \eqref{eq:w2toshow}. Therefore, Proposition~\ref{prop:stability_weaker} implies the assertion.
\end{proof}

\subsection*{Acknowledgments}
	M. Passenbrunner is supported by the Austrian Science Fund FWF, project P32342.

\bibliographystyle{plain}
\bibliography{local}

\end{document}